\documentclass{amsart}

\usepackage{amssymb}
\usepackage{amscd}
\usepackage[all,cmtip]{xy}
\usepackage{hyperref}
\usepackage{graphicx}

\newcommand{\comment}[1]{}

\newcommand{\F}{\mathbf{F}}

\newcommand{\Q}{\mathbf{Q}}

\newcommand{\Z}{\mathbf{Z}}

\newcommand{\m}{\mathfrak{m}}
\newcommand{\na}{\mathfrak{n}}
\newcommand{\pa}{\mathfrak{p}}
\newcommand{\qa}{\mathfrak{q}}

\newcommand{\sep}{\mathrm{sep}}
\newcommand{\ins}{\mathrm{ins}}

\newcommand{\iv}{\mathrm{i}}
\newcommand{\hv}{\mathrm{h}}
\newcommand{\vv}{\mathrm{v}}
\newcommand{\ww}{\mathrm{w}}
\newcommand{\tv}{\mathrm{t}}
\newcommand{\sv}{\mathrm{s}}
\newcommand{\wv}{\mathrm{w}}
\newcommand{\kv}{\mathrm{k}}

\newcommand{\fod}{\mathbin{|\hspace{-0.07cm}\backslash}}

\DeclareMathOperator{\tr}{tr}

\DeclareMathOperator{\im}{im}

\DeclareMathOperator{\Gal}{Gal}

\DeclareMathOperator{\Hom}{Hom}

\DeclareMathOperator{\Aut}{Aut}
\DeclareMathOperator{\Spec}{Spec}
\DeclareMathOperator{\MaxSpec}{MaxSpec}

\DeclareMathOperator{\K}{K}

\DeclareMathOperator{\expo}{exp}
\DeclareMathOperator{\lcm}{lcm}
\DeclareMathOperator{\ord}{ord}
\DeclareMathOperator{\fv}{f}
\DeclareMathOperator{\ev}{e}
\DeclareMathOperator{\nv}{n}
\DeclareMathOperator{\dv}{d}
\DeclareMathOperator{\gv}{g}
\DeclareMathOperator{\pv}{p}
\DeclareMathOperator{\Dv}{D}
\DeclareMathOperator{\Iv}{I}
\DeclareMathOperator{\Vv}{V}
\DeclareMathOperator{\chart}{char}
\DeclareMathOperator{\trdeg}{trdeg}

\DeclareMathOperator{\cd}{cd}

\newcommand{\limplies}{\Longleftarrow}

\theoremstyle{plain}
\newtheorem{theorem}{Theorem}[section]

\newtheorem{corollary}[theorem]{Corollary}

\newtheorem{lemma}[theorem]{Lemma}

\newtheorem{proposition}[theorem]{Proposition}

\theoremstyle{definition}
\newtheorem{definition}[theorem]{Definition}
\newtheorem{remark}[theorem]{Remark}

\newtheorem{example}[theorem]{Example}

\begin{document}

\title{The algebraic theory of valued fields}
\author{Michiel Kosters}
\address{Mathematisch Instituut
P.O. Box 9512
2300 RA Leiden
the Netherlands}
\email{mkosters@math.leidenuniv.nl}
\urladdr{www.math.leidenuniv.nl/~mkosters}
\date{\today}
\thanks{This article is part of my PhD thesis written under the supervision of Hendrik Lenstra. I would like to congratulate him with his $65$th birthday and as a present give him the reference he subtly requested.}
\subjclass[2010]{13A18, 12J20, 12J10}

\begin{abstract}
In this exposition we discuss the theory of algebraic extensions of valued fields. Our approach is mostly through Galois theory. Most of the results are well-known, but some are new. No previous knowledge on the theory of valuations is needed.
\end{abstract}

\maketitle
\tableofcontents

\section{Introduction}

General valuation theory plays an important role in many areas in mathematics. Also in this thesis, we will quite often need valuation theory,
although for our applications the theory of discrete valuations suffices. There exist many books on valuation theory, such as \cite{END}, \cite{ENG},
\cite{KUH} and \cite{EFR}. They do not treat the case of algebraic extensions of valuations theory completely. Furthermore, definitions of certain
concepts are not uniform. This article is written to fill this gap in the literature and provide a useful reference, even when restricting to
the case of discrete valuations. Our definitions are motivated by our Galois theoretic approach. No previous knowledge on
the theory of valuations is needed and only a slight proficiency in commutative algebra suffices (see for example \cite{AT}
and \cite{LA}).

With this in mind, this article starts with definitions and the main statements. In the second part of this article we will provide complete proofs.
In the last part of this article we give examples of extensions with a defect.

Our treatment of valuation theory starts with normal extensions of valued fields. Later, by looking at group actions on fundamental sets, we prove
statements for algebraic extensions of valued fields. The beginning of our Galois-theoretic approach follows parts of \cite{END} and \cite{ENG},
although we prove that certain actions are transtive in a different way. The
upcoming book \cite{KUH} uses at certain points a very
similar approach.

Even though most of the statements in this article are known, there are a couple of new contributions.

\begin{itemize}
 \item We define when algebraic extensions of valued fields are \emph{immediate}, \emph{unramified}, \emph{tame}, \emph{local}, \emph{totally
ramified} or \emph{totally wild} (Definition \ref{1c9000}). The definitions are motivated by practicality coming from Galois theory.
We also study maximal respectively minimal extensions with these properties (Theorem \ref{1c803}).
 \item We compute several quantities, such as separable residue field degree extension, tame ramification index and more in finite algebraic
extensions of valued fields in terms of automorphism groups (Proposition \ref{1c854}). We will give necessary and
sufficient conditions for algebraic extensions of valued fields to be immediate, unramified, \ldots in terms of automorphism groups and
fundamental sets (Theorem \ref{1c802}). 
Current literature only seems to handle the Galois case.
\item Another notable result is Theorem \ref{1c7776}: classical sequences from valuation theory split if a specific residue field has some properties. 
\end{itemize}

For a field $K$ we denote by $\overline{K}$ an algebraic closure\index{$\overline{K}$}. For a domain $R$ we denote by $Q(R)$\index{$Q(R)$} its field
of fractions.

\section{Definition of valuations}

Let $K$ be a field. 

\begin{definition}
A \emph{valuation ring}\index{valuation ring} on $K$ is a subring $\mathcal{O} \subseteq K$ such that for all $x \in K^*$ we have
$x \in \mathcal{O}$ or $x^{-1} \in \mathcal{O}$. 
\end{definition}

\begin{lemma}
 There is a bijection between the set of valuation rings of $K$ and the set of relations $\leq$ on $K^*$ which satisfy for $x,y,z \in K^*$ 
\begin{enumerate}
 \item $x \leq y$ or $y \leq x$;
 \item $x \leq y$, $y \leq z$ $\implies$ $x \leq z$;
 \item $x \leq y$ $\implies$ $xz \leq yz$;
 \item if $x+y \neq 0$, then $x \leq x+y$ or $y \leq x+y$.
\end{enumerate}
This bijection maps a valuation ring $\mathcal{O}$ to the relation which for $x, y \in K^*$ is defined by: $x \leq y$ iff $y/x \in \mathcal{O}$. The
inverse maps $\leq$ to $\{x \in K^*: 1 \leq x\} \sqcup \{0\}$.
\end{lemma}
\begin{proof}
 Let $\mathcal{O}$ be a valuation ring and consider the obtained relation $\leq$. Then i holds by definition. Property ii, iii hold as
$\mathcal{O}$ is a ring. For iv, suppose that $x \leq y$, that is, $y/x \in \mathcal{O}$. Then we have $1+y/x=(x+y)/x \in \mathcal{O}$. Hence $x \leq
x+y$ as required.

Given $\leq$, we claim that $\mathcal{O}=\{x \in K^*: 1 \leq x\} \sqcup \{0\}$ is a valuation ring. Let $x \in K^*$. We have $1 \leq 1$ (i) and
hence $1 \in \mathcal{O}$. Furthermore,
$-1 \in \mathcal{O}$. Indeed, by i we have $1 \leq -1$ or $-1 \leq 1$. In the first case we are done, in the second case we can multiply by $-1$
to obtain $1 \leq -1$ (iii). Take $x, y \in \mathcal{O} \setminus \{0\}$. Then if we multiply $x \geq 1$ by $y$ we obtain $xy \geq y \geq 1$ (iii),
and hence we have $xy \in \mathcal{O}$ (ii). If $x+y \neq 0$, we find $x+y\geq x \geq 1$ or $x+y \geq y \geq 1$. From ii we conclude that $x+y
\geq 1$. Take $z \in K^*$. Then we have
Finally, we have $1 \leq z$ or $z \leq 1$ (i). In the first case, we have $z \in \mathcal{O}$. In the second case, we multiply by $z^{-1}$ and iv
gives $1 \leq z^{-1}$. Hence $z^{-1} \in \mathcal{O}$. This shows that $\mathcal{O}$ is a valuation ring.
\end{proof}

Let $\mathcal{O}$ be a valuation on $K$. Consider the relation $\leq$ on $K^*$ induced from $\mathcal{O}$ as in the lemma above. One easily sees that
$\mathcal{O}^*=\{x \in K^*: 1 \leq x \textrm{ and } x \leq 1\}$. Furthermore, if $x, y \in \mathcal{O} \setminus \mathcal{O}^*$, we deduce from
property iv and ii that $x+y$ is not a unit. Hence $\mathcal{O}$
is a local ring. The induced relation $\leq$ on $K^*$ makes $K^*/\mathcal{O}^*$ into an \emph{ordered abelian group}\index{ordered abelian group}.
An ordered abelian group is an abelian group $P$, written additively, together with a relation $\leq$ such that for $a, b, c \in P$
we have:
\begin{enumerate}
 \item $a \leq b$, $b \leq a$ $\implies$ $a=b$;
 \item $a \leq b$, $b \leq c$ $\implies$ $a \leq c$;
 \item $a \leq b$ or $b \leq a$;
 \item $a \leq b$ $\implies$ $a+c \leq b+c$.
\end{enumerate}
The group morphism $v \colon K^* \to K^*/\mathcal{O}^*$ is
called the \emph{valuation map}\index{valuation map} and it satisfies for $x, y \in K^*$ with $x+y \neq 0$: $v(x+y) \geq \min(v(x),v(y))$. The ordered
abelian group $K^*/\mathcal{O}^*$ is called the \emph{value group}\index{valuation!value group}. 

To shorten notation we just write $v$ for a valuation. We denote by $\mathcal{O}_v$\index{$\mathcal{O}_v$} the
valuation ring with maximal ideal $\m_v$\index{$\m_v$}. The residue field is denoted by $\kv_v=\mathcal{O}_v/\m_v$\index{$\kv_v$}. The value group is denoted by
$\Delta_v=K^*/\mathcal{O}^*$\index{$\Delta_v$}, for which we use
additive notation. We use the notation $v \colon K^* \to \Delta_v$. 
We set\index{$\pv_v$} 
\begin{align*}
 \pv_v= \left\{ \begin{array}{cc} \chart(\kv_v) &\textrm{if } \chart(\kv_v) \neq 0 \\
         1 &\textrm{if } \chart(\kv_v) = 0.
        \end{array} \right.
\end{align*}
 A pair $(K,v)$ as above is called a \emph{valued field}\index{valued field}. Note that a valuation $v$ gives rise to the valued field
$(Q(\mathcal{O}_v),v)$ where $Q(\mathcal{O}_v)$ is the fraction field of $\mathcal{O}_v$. If $K'$ is a subfield, then we denote by $v|_{K'}$\index{$v\vert_{K'}$}\index{valued field!restriction} the
valuation on $K'$ corresponding to the valuation ring $\mathcal{O}_v \cap K$.

\section{Main results}

In this section we will provide statements of the main results. Proofs of the statements follow in Sections 4, 5, 6 and 7 and will occupy most of
this article.

\subsection{Properties of extensions of valuations}

Let $M \supseteq N$ be an extension of field. When we say that $M/N$ is separable we mean that it is algebraic and separable. Similarly, normal means
normal and algebraic (but not necessarily separable). Assume that $M/N$ is finite. We set $[M:N]_{\sv}$ for the separability
degree of the extension and $[M:N]_{\iv}$ for the
inseparability degree\index{$[M:N]_{\sv}$}\index{$[M:N]_{\iv}$}. Note that $[M:N]=[M:N]_{\sv} \cdot [M:N]_{\sv}$. 

Let $(K,v)$ be a valued field and let $L$ be an extension of $K$. An \emph{extension}\index{valued
field!extension} of $v$ to $L$ is a valuation $w$ on $L$ satisfying $\mathcal{O}_w \cap K=\mathcal{O}_v$, equivalently,
$\m_w \cap K=\m_v$. Such extensions do exist (Proposition \ref{1c55}). We denote such an extension by $(L,w) \supseteq (K,v)$\index{$(L,w) \supseteq (K,v)$} or $w/v$\index{$w/v$}. Sometimes we write $w|v$ if $w$ extends $v$. The number of
extensions of $v$ to $L$ is denoted by
$\gv_{L,v}$\index{$\gv_{L,v}$}, which is finite if $L/K$ is finite (Proposition
\ref{1c55}). Such an extension $(L,w) \supseteq (K,v)$ is called \emph{finite}\index{valued field!extension!finite} if $L/K$ is finite. In a
similar way we define such an extension to be normal, separable, \ldots. An extension induces inclusions $\Delta_v \to \Delta_w$ and $\kv_v \to
\kv_w$. The following proposition defines a lot of quantities relating to a finite extension of valued fields and gives some properties of these
quantities (see Proposition \ref{1c4444}).

\begin{proposition} \label{1c9871}
Let $(L,w) \supseteq (K,v)$ be a finite extension of valued fields. Then one has: 
\begin{itemize}
 \item $\ev(w/v):=(\Delta_w:\Delta_v) \in \Z_{\geq 1}$ (\emph{ramification index}\index{valued field!extension!ramification
index}\index{$\ev(w/v)$}); 
 \item $\ev_{\tv}(w/v):=\lcm\{m \in \Z_{\geq 1}: m|\ev(w/v),\ \gcd(m,\pv_v)=1\} \in \Z_{\geq 1}$ (\emph{tame ramification index}\index{valued
field!extension!tame ramification index}\index{$\ev_{\tv}(w/v)$}); 
 \item $\ev_{\ww}(w/v):=\frac{\ev(w/v)}{\ev_{\tv}(w/v)} \in \pv_v^{\Z_{\geq 0}}$ (\emph{wild ramification index}\index{valued field!extension!wild
ramification index}\index{$\ev_{\ww}(w/v)$});
 \item $\fv(w/v):=[\kv_w:\kv_v] \in \Z_{\geq 1}$ (\emph{residue field degree}\index{valued field!extension!residue field degree}\index{$\fv(w/v)$}); 
 \item $\fv_{\sv}(w/v):=[\kv_w:\kv_v]_{\sv} \in \Z_{\geq 1}$ (\emph{separable residue field degree}\index{valued field!extension!separable residue
field
degree}\index{$\fv_{\sv}(w/v)$}); 
 \item $\fv_{\iv}(w/v):=[\kv_w:\kv_v]_{\iv} \in \pv_v^{\Z_{\geq 0}}$ (\emph{inseparable residue field degree}\index{valued field!extension!inseparable
residue field degree}\index{$\fv_{\iv}(w/v)$});
 \item Let $M/K$ be a finite normal extension containing $L$. We define the \emph{local degree}\index{valued field!extension!local
degree}\index{$\nv(w/v)$} by $\nv(w/v):=
\frac{\gv_{M,w}}{\gv_{M,v}} \cdot [L:K] \in \Z_{\geq 1}$ and this does not depend on the choice of $M$;
 \item $\dv(w/v):=\frac{\nv(w/v)}{\ev(w/v)\fv(w/v)} \in \pv_v^{\Z_{\geq 0}}$ (\emph{defect}\index{valued field!extension!defect}\index{$\dv(w/v)$}); 
 \item $\dv_{\wv}(w/v):=\dv(w/v)\ev_{\ww}(w/v)\fv_{\iv}(w/v) \in \pv_v^{\Z_{\geq 0}}$ (\emph{wildness index}\index{valued
field!extension!wildness index}\index{$\dv_{\wv}(w/v)$});
\end{itemize}
The quantities $\ev$, $\ev_{\tv}$, $\ev_{\ww}$, $\fv$, $\fv_{\sv}$, $\fv_{\iv}$, $\nv$, $\dv$ and $\dv_{\wv}$ are multiplicative in towers.
\end{proposition}

\begin{definition} \label{1c9000}
Let $(L,w) \supseteq (K,v)$ be a finite extension of valued fields. Then we have the following properties which $(L,w) \supseteq (K,v)$ can satisfy:
\begin{itemize}
 \item \emph{immediate}\index{valued field!extension!immediate}: $\dv_{\wv}(w/v)=\ev_{\tv}(w/v)=\fv_{\sv}(w/v)=1$, equivalently, $\nv(w/v)=1$;
 \item \emph{unramified}\index{valued field!extension!unramified}: $\dv_{\wv}(w/v)=\ev_{\tv}(w/v)=1$;
 \item \emph{tame}\index{valued field!extension!tame}: $\dv_{\wv}(w/v)=1$; 
 \item \emph{local}\index{valued field!extension!local}: $\gv_{L,v}=1$;
 \item \emph{totally ramified}\index{valued field!extension!totally ramified}: $\fv_{\sv}(w/v)=\gv_{L,v}=1$;
 \item \emph{totally wild}\index{valued field!extension!totally wild}: $\ev_{\tv}(w/v)=\fv_{\sv}(w/v)=\gv_{L,v}=1$.
\end{itemize}
We say that $v$ is \emph{totally split}\index{valued field!extension!totally split} in $L$ if $g_{L,v}=[L:K]$. 

As the various degrees are multiplicative, we can extend this definition in the following way.
Let $(L,w) \supseteq (K,v)$ be an algebraic extension of valued fields. Then $w/v$ is \emph{immediate} (respectively \emph{unramified}, \emph{tame},
\emph{local}, \emph{totally
ramified}, \emph{totally wild}) if all intermediate extensions $(L',w')$ of $(L,w) \supseteq (K,v)$ where $L'/K$ is finite are immediate (respectively
unramified, tame,
local, totally ramified, totally wild). We say that $v$ is \emph{totally split} in $L$ if all intermediate extensions $(L',w')$ of
$(L,w) \supseteq (K,v)$ with $L'/K$ finite are totally split.
\end{definition}

\subsection{Normal extensions}

\begin{definition} \label{1c344}
 Let $(M,x) \supseteq (K,v)$ be a normal algebraic extension of valued fields and let $G=\Aut_K(M)$. Note that $G$ acts on the set of valuations on
$M$ extending $v$ by $\mathcal{O}_{g(x')}=g(\mathcal{O}_{x'})$. Let \index{$\Dv_{x,K}$}$\Dv_{x,K}=\{g \in G:
gx=x\}$ be the \emph{decomposition group}\index{decomposition group}\index{valued field!extension!decomposition group} of $x$ over $K$. We define the
\emph{inertia
group}\index{inertia group}\index{valued field!extension!inertia group}\index{$\Iv_{x,K}$} $\Iv_{x,K} \subseteq \Dv_{x,K}$ of $x$ over $K$ to be the
kernel of the
natural group morphism $\Dv_{x,K} \to \Aut_{\kv_v}(\kv_x)$. Furthermore, there is a
natural group morphism 
\begin{align*}
\Iv_{x,K} &\to \Hom(\Delta_x/\Delta_v,\kv_x^*) \\
\overline{c} &\mapsto \overline{\frac{g(c)}{c}}
\end{align*}
(see Lemma \ref{1crene}). We define the
\emph{ramification group}\index{ramification group}\index{valued
field!extension!ramification group}\index{$\Vv_{x,K}$} of $x$ over $K$ to be its kernel. We denote it by $\Vv_{x,K}$. 

Let \index{$\Gamma_{x,v}$}$\Gamma_{x,v} = \im \left( \Aut_{\kv_v}(\kv_x) \to \Aut_{\kv_v^*}(\kv_x^*) \right)$.
Let $\Aut_{K^*,\Gamma_{x,v}}(M^*/(1+\m_x))$\index{$\Aut_{K^*,\Gamma_{x,v}}(M^*/(1+\m_x))$} be the subgroup of $\Aut(M^*/(1+\m_x))$ consisting of those automorphisms such that the restriction to $\kv_x^*$ lies in $\Gamma_{x,v}$ and which are the identity on $K^*/(1+\m_v)$. We have a natural map
$\Dv_{x,K} \to \Aut_{K^*,\Gamma_{x,v}}(M^*/(1+\m_x))$ (see the discussion after Lemma \ref{1crene}).
\end{definition}

We endow $G$ with the profinite topology. This means that we view $G$ as a subset of $M^M$. We endow $M$ with the discrete topology, $M^M$
with the product topology and $G$ with the induced topology. Similarly we define profinite topologies on $\Aut_{\kv_v}(\kv_x) \subseteq
\kv_x^{\kv_x}$ and $\Hom(\Delta_x/\Delta_v, \kv_x^*) \subseteq \left(\kv_x^* \right)^{\Delta_x/\Delta_v}$ where $\kv_x$ and $\kv_x^*$ have the
discrete topology. 
Furthermore, let $S$ be the set of valuations extending $v$ to $M$. For $x' \in S$ and a finite extension $L$ of $K$ in $M$ we set $U_{x',L}=\{x''
\in S: x''|_L=x'|_L\}$. This is a basis for a topology on $S$. 
We give $\Aut_{K^*,\Gamma_{x,v}}(M^*/(1+\m_x))$ the following topology. We give $C=M^*/(1+\m_x)$ the discrete topology, $C^C$ the product topology and $\Aut_{K^*,\Gamma_{x,v}}(M^*/(1+\m_x))$ the induced topology.

\begin{definition}
 Let $L/K$ be a field extension. We set $L_{K,\sep}$\index{$L_{K,\sep}$} for the field extension of $K$ consisting of
the elements in $L$ which are separable and algebraic over $K$. 
\end{definition}

\begin{definition}
Let $(M,x) \supseteq (K,v)$ be a normal algebraic extension of valued fields. We define $K_{\hv,x}=M_{K,\sep}^{\Dv_{x,K}}$ (\emph{decomposition
field}\index{decomposition field}\index{$K_{\hv,x}$}, $\hv$ stands for Henselization), $K_{\iv,x}=M_{K,\sep}^{\Iv_{x,K}}$ (\emph{inertia
field}\index{inertia field}\index{$K_{\iv,x}$})
and $K_{\vv,x}=M_{K,\sep}^{\Vv_{x,K}}$ (\emph{ramification field}\index{ramification field}\index{$K_{\vv,x}$}). Note that all these extensions are
separable over $K$ and that we have $K \subseteq K_{\hv,x} \subseteq K_{\iv,x} \subseteq K_{\vv,x} \subseteq M$.
\end{definition}

Recall that for a prime $p$ and a profinite group $H$ a pro-$p$-Sylow subgroup $H'$ is a maximal subgroup of $H$ such that $H'$ is a projective
limit of finite groups of $p$-power order. 

We define the \emph{Steinitz monoid}\index{Steinitz monoid} as the following set. Let $\mathcal{P} \subset \Z$ be the set of primes. \emph{Steinitz numbers}\index{Steinitz numbers} are of the form
$\prod_{p \in \mathcal{P}} p^{n_p}$ with $n_{p} \in \Z_{\geq 0} \sqcup \{\infty\}$. This set has an obvious monoid structure and there is an obvious way for defining $\gcd$ and $\lcm$ for arbitrary sets of Steinitz numbers. Furthermore, there is an obvious notion of divibility. 

Let $H$ be a profinite group. Then we define its \emph{order}\index{order}\index{$\ord(H)$}\index{profinite group!order} to be
\begin{align*}
 \ord(H) = \lcm \{ [H:N]: N \textrm{ open normal subgroup of }H\},
\end{align*}
and we define its \emph{exponent}\index{exponent}\index{$\expo(H)$}\index{profinite group!exponent} to be
\begin{align*}
 \expo(H) = \lcm \{ \expo(H/N): N \textrm{ open normal subgroup of }H\}.
\end{align*}
Both are Steinitz numbers. Furthermore, if $H=\underset{\underset{i \in I}{\leftarrow}}{\lim}\ H_i$ where the $H_i$ are finite, then one has $\ord(H)=\lcm(\ord(H_i): i \in I)$ and $\expo(H)=\lcm(\exp(H_i): i \in I)$. 

The proof of the following theorem can be found on Page \pageref{1cp1}.

\begin{theorem} \label{1c333}
 Let $(M,x) \supseteq (K,v)$ be a normal algebraic extension of valued fields and let $G=\Aut_K(M)$. Then $G$ acts continuously on the set $S$ consisting of the valuations of $M$ extending $K$ and induces an isomorphism of topological $G$-sets
\begin{align*}
 G/\Dv_{x,K} &\to S \\
\overline{g} &\mapsto gx. 
\end{align*}
For $g \in G$ one has $\Dv_{g(x),K}=g \Dv_{x,K} g^{-1}$, $\Iv_{g(x),K}=g \Iv_{x,K} g^{-1}$ and $\Vv_{g(x),K}=g \Vv_{x,K} g^{-1}$. One also has $K_{\hv,g(x)}= g K_{\hv,x}$, $K_{\iv,g(x)}= g K_{\iv,x}$ and $K_{\vv,g(x)}= g K_{\vv,x}$.

Furthermore, we have exact sequences of profinite groups
\begin{align*}
 0 \to \Iv_{x,K} \to \Dv_{x,K} \to \Aut_{\kv_v}(\kv_x) \to 0,
\end{align*}
\begin{align*}
 0 \to \Vv_{x,K} \to \Iv_{x,K} \to \Hom(\Delta_x/\Delta_v,\kv_x^*) \to 0
\end{align*}
and
\begin{align*}
 0 \to \Vv_{x,K} \to \Dv_{x,K} \to  \Aut_{K^*,\Gamma_{x,v}}(M^*/(1+\m_x)) \to 0.
\end{align*}
The extension $\kv_x/\kv_v$ is normal and $\Vv_{x,K}$ is the unique pro-$\pv_{v}$-Sylow subgroup of
$\Iv_{x,K}$. Then for any integer $r \in \Z_{\geq 1}$ dividing $\exp(\Iv_{x,K}/\Vv_{x,K})$ the field $\kv_x$ contains a primitive $r$-th root of unity.

Let $(L,w)$ be an intermediate extension of $(M,x) \supseteq (K,v)$ and let $H=\Aut_L(M)$. Then one has:
\begin{enumerate}
 \item $\Dv_{x,L}=\Dv_{x,K} \cap H$, $\Iv_{x,L}=\Iv_{x,K} \cap H$ and $\Vv_{x,L}=\Vv_{x,K} \cap H$; 
 \item $L_{\hv,x}=K_{\hv,x}L$, $L_{\iv,x}=K_{\iv,x}L$ and $L_{\vv,x}=K_{\vv,x}L$.
\end{enumerate}
If in addition we assume that $L/K$ is normal, then we have exact sequences
\begin{align*}
0 \to \Dv_{x,L} \to \Dv_{x,K} \to \Dv_{w,K} \to 0,
\end{align*}
\begin{align*}
0 \to \Iv_{x,L} \to \Iv_{x,K} \to \Iv_{w,K} \to 0,
\end{align*} and
\begin{align*}
0 \to \Vv_{x,L} \to \Vv_{x,K} \to \Vv_{w,K} \to 0.
\end{align*}
Under the normality assumption we have $K_{\hv,x|_{L}}=K_{\hv,x} \cap L$, $K_{\iv,x|_{L}}=K_{\iv,x} \cap L$ and $K_{\vv,x|_{L}}=K_{\vv,x} \cap L$.
\end{theorem}

If the extension $M/K$ is finite, the previous theorem implies the following (proof on Page \pageref{1cp3}).

\begin{proposition} \label{1c854}
Let $(M,x) \supseteq (K,v)$ be a finite normal extension of valued fields. Then one has
\begin{align*}
~ [K_{\hv,x}:K] &= \gv_{M,v} \\
~ [K_{\iv,x}:K_{\hv,x}] &= \fv_{\sv}(x/v)=\fv_{\sv}(x|_{K_{\iv,x}}/x|_{K_{\hv,x}}) \\
~ [K_{\vv,x}:K_{\iv,x}] &= \ev_{\tv}(x/v)=\ev_{\tv}(x|_{K_{\vv,x}}/x|_{K_{\iv,x}}) \\
~ [M:K_{\vv,x}] &= \dv_{\wv}(x/v)=\dv_{\wv}(x/x|_{K_{\vv,x}}).
\end{align*}
Let $(L,w)$ be an intermediate extension of $(M,x) \supseteq (K,v)$. Then one has
\begin{align*}
~ [L_{\hv,x}:K_{\hv,x}] &= \dv_{\wv}(w/v)\ev_{\tv}(w/v)\fv_{\sv}(w/v) \\
~ [L_{\iv,x}:K_{\iv,x}] &= \dv_{\wv}(w/v)\ev_{\tv}(w/v) \\
~ [L_{\vv,x}:K_{\vv,x}] &= \dv_{\wv}(w/v).
\end{align*}
\end{proposition}

The proof of the following theorem can be found on Page \pageref{1c7779}.

\begin{theorem} \label{1c7776}
Let $(M,x) \supseteq (K,v)$ be a normal extension of valued fields. Then the following hold.
\begin{enumerate}
\item Assume that $\kv_x$ has no cyclic extensions of prime order dividing the order of $\Iv_{x,K}/\Vv_{x,K}$. Then the exact sequence 
\begin{align*}
0 \to \Iv_{x,K}/\Vv_{x,K} \to \Dv_{x,K}/\Vv_{x,K} \to \Dv_{x,K}/\Iv_{x,K} \to 0
\end{align*}
is right split.
\item Assume that $\kv_x$ has no cyclic extensions of prime order dividing $\pv_v$ or that $\pv_v \nmid \ord(\Iv_{x,K})$. Then the exact sequence 
\begin{align*}
0 \to \Vv_{x,K} \to \Dv_{x,K} \to \Dv_{x,K}/\Vv_{x,K} \to 0
\end{align*}
is right split.
\item Assume that $\kv_x$ has no cyclic extensions of prime order dividing $\ord(\Iv_{x,K})$. Then the exact sequence
\begin{align*}
0 \to \Iv_{x,K} \to \Dv_{x,K} \to \Dv_{x,K}/\Iv_{x,K} \to 0
\end{align*}
is right split.
\end{enumerate}
\end{theorem}

\subsection{Algebraic extensions}

A well-known result in the following (proof on Page \pageref{1cp4}).

\begin{theorem}[Fundamental equality] \label{1c933} \index{fundamental equality}\index{valued field!extension!fundamental equality}
 Let $(K,v)$ be a valued field and let $L/K$ be a finite field extension. Then we have 
\begin{align*}
[L:K] &= \sum_{w|v \textrm{ on }L} \nv(w/v) = \sum_{w|v \textrm{ on }L} \dv(w/v)\ev(w/v)\fv(w/v) \\
 &\geq \sum_{w|v  \textrm{ on }L}
\ev(w/v)\fv(w/v). 
\end{align*}
\end{theorem}

Two algebraic field extensions $L, L'$ of a field $K$ are called \emph{linearly
disjoint}\index{linearly disjoint}\index{field extension!linearly disjoint} over $K$ if $L \otimes_K L'$ is a field.

If $L, L'$ are subfields of a field $\Omega$, then we set the \emph{compositum}\index{compositum} $LL'=\im(L \otimes_{\Z} L' \to \Omega)$. This
is the smallest ring containing both $L$ and $L'$ in $\Omega$. This is a field if the elements of $L$ are algebraic over $L'$ or if the elements of
$L'$ are algebraic over $L$.

The following proposition studies extensions of valuations using fundamental sets (Proof on \pageref{1cp5}). If $L \supseteq K$ and $M \supseteq K$ are extensions of fields, we denote by $\Hom_K(L,M)$\index{$\Hom_K(L,M)$} the set of field homomorphisms from $L$ to $M$ which are the identity on $K$. 

\begin{theorem} \label{1c802} 
Let $(K,v)$ be a valued field and let $L/K$ be an algebraic extension. Let $(M,x) \supseteq (K,v)$ be a normal extension with group $G=\Aut_K(M)$ such
that the $G$-set $X=\Hom_K(L,M)$ is not empty. Then the natural map
\begin{align*}
 \pi \colon X &\to \{w \textrm{ on }L \textrm{ extending } v\} \\
\sigma &\mapsto w \textrm{ s.t. } \mathcal{O}_w=\sigma^{-1} ( \mathcal{O}_x \cap \sigma(L))
\end{align*}
is surjective. Let $\sigma \in X$ and set $w=\pi(\sigma)$ and let $G_{\sigma}$ be the stabilizer in $G$ of $\sigma$. Then we have:
\begin{enumerate}
 \item $w/v$ is immediate $\iff$ $\sigma(L) \subseteq K_{\hv,x}$ $\iff$ $\Dv_{x,K} \subseteq G_{\sigma}$;
 \item $w/v$ is unramified $\iff$ $\sigma(L) \subseteq K_{\iv,x}$ $\iff$ $\Iv_{x,K} \subseteq G_{\sigma}$;
 \item $w/v$ is tame $\iff$ $\sigma(L) \subseteq K_{\vv,x}$ $\iff$ $\Vv_{x,K} \subseteq G_{\sigma}$; 
 \item $w/v$ is local $\iff$ $\sigma(L)$ and $K_{\hv,x}$ are linearly disjoint over $K$ $\iff$ $\Dv_{x,K} \sigma=X$;
 \item $w/v$ is totally ramified $\iff$ $\sigma(L)$ and $K_{\iv,x}$ are linearly disjoint over $K$ $\iff$ $\Iv_{x,K}\sigma=X$;
 \item $w/v$ is totally wild $\iff$ $\sigma(L)$ and $K_{\vv,x}$ are linearly disjoint over $K$ $\iff$ $\Vv_{x,K} \sigma=X$.
\end{enumerate}
Furthermore we have:
\begin{enumerate}
\setcounter{enumi}{6}
 \item $x/w$ is immediate $\iff$ $M=\sigma(L)K_{\hv,x}$ $\iff$ $M/\sigma(L)$ is separable and $G_{\sigma} \cap \Dv_{x,K}=0$;
 \item $x/w$ is unramified $\iff$ $M=\sigma(L)K_{\iv,x}$ $\iff$  $M/\sigma(L)$ is separable and $G_{\sigma} \cap \Iv_{x,K}=0$;
 \item $x/w$ is tame $\iff$ $M=\sigma(L)K_{\vv,x}$ $\iff$  $M/\sigma(L)$ is separable and $G_{\sigma} \cap \Vv_{x,K}=0$;
 \item $x/w$ is local $\iff$ $\sigma(L) \supseteq K_{\hv,x}$ $\iff$ $G_{\sigma} \subseteq \Dv_{x,K}$;
 \item $x/w$ is totally ramified $\iff$ $\sigma(L) \supseteq K_{\iv,x}$ $\iff$ $G_{\sigma} \subseteq \Iv_{x,K}$;
 \item $x/w$ is totally wild $\iff$ $\sigma(L) \supseteq K_{\vv,x}$ $\iff$ $G_{\sigma} \subseteq \Vv_{x,K}$.
\end{enumerate}
Finally we have:
\begin{enumerate}
\setcounter{enumi}{12}
 \item $v$ is totally split in $L$ $\iff$ for all $\sigma \in X$ we have $\sigma(L) \subseteq K_{\hv,x}$ $\iff$ $\Dv_{x,K}$
acts trivially on $X$;
 \item $w$ is totally split in $M$ $\iff$ $M/\sigma(L)$ is separable and only the trivial element of $G_{\sigma}$ is conjugate to an element of $\Dv_{x,K}$.
\end{enumerate}
\end{theorem}

The above proposition has a lot of corollaries. The proof of the first corollary can be found on Page \pageref{1cp6}.

\begin{corollary} \label{1c1432}
Let $(K,v)$ be a valued field and let $L$ and $L'$ be two algebraic extensions of $K$ in some algebraic closure of $K$. Let $x$ be a valuation on
$LL'$
extending $v$ and let $w=x|_L$ and $w'=x|_{L'}$. Then the following statements hold:
\begin{enumerate}
 \item if $w/v$ is immediate, then $x/w'$ is immediate; 
 \item if $w/v$ is unramified, then $x/w'$ is unramified;
 \item if $w/v$ is tame, then $x/w'$ is tame;
 \item if $v$ is totally split in $L$, then $w'$ is totally split in $LL'$. 
\end{enumerate}
\end{corollary}

The proof of the following corollary can be found on Page \pageref{1cp7}.

\begin{corollary} \label{1c842}
 Let $(L,w) \supseteq (K,v)$ be an algebraic extension of valued fields and let $(K',w')$ be an intermediate extension. Then $w/v$ is immediate
(respectively
unramified, tame, local, totally ramified, totally wild) iff $w/w'$ and $w'/v$ are immediate (respectively unramified, tame, local, totally
ramified, totally wild).
\end{corollary}

The proof of the following proposition can be found on Page \pageref{1cp2}.

\begin{proposition} \label{1c432}
 Let $\Omega$ be a field and let $L,L' \subseteq \Omega$ be subfields. Then there is a subfield $M$ of $L$ such that for all subfields $F$
of $L$ the natural map $L \otimes_F (L'F) \to LL'$ is an isomorphism
iff $M \subseteq F$. Furthermore, $M$ can be described in the following two ways, where $\F$ is the prime of field of $\Omega$.
\begin{enumerate}
 \item Let $\mathfrak{B} \subseteq L'$ be a basis of $LL'$ over $L$. For $b \in \mathfrak{B}$ and $x \in L'$ write $x=\sum_{b \in
\mathfrak{B}}c_{x,b}b$ with $c_{x,b} \in L$ almost all zero. Then one has $M=\F(c_{x,b}: x \in L',\ b \in \mathfrak{B})$. 
 \item Set	
\begin{align*}
S &=\{c \in L: \exists I \subseteq L' \textrm{ finite, ind. over }L \textrm{ and } (c_i)_{i \in I} \in L^I \\
&\ \ \ \textrm{ s.t. } \exists i \in I
	\textrm{ s.t.
}c_i=c \textrm{ and } \sum_{i \in I} c_i i \in L'\}.
\end{align*}
Then one has $M=\F(S)$.
\end{enumerate}
\end{proposition}

\begin{definition}
 The field $M$ in the above theorem is called the \index{field of definition}\emph{field of definition} of $L'$ over $L$ and is denoted by $L \fod
L'$\index{$\fod$}. 
\end{definition}

The proof of the following theorem can be found on Page \pageref{1cp8}.

\begin{theorem} \label{1c803}
 Let $(L,w) \supseteq (K,v)$ be an algebraic extension of valued fields. Then then following statements hold:
\begin{enumerate}
 \item There is a unique maximal subextension $L_1$ of $L/K$ such that $w|_{L_1}/v$ is immediate (respectively $L_2$ for unramified and $L_3$ for
tame).
 \item There is a unique minimal subextension $L_4$ of $L/K$ such that $w/w|_{L_4}$ is local (respectively $L_5$ for totally ramified and $L_6$ for
totally wild).
\end{enumerate}
We have the following diagram of inclusions:
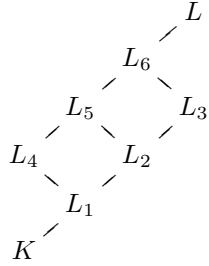
\[
\xymatrix@=0.4em{
 & & & L\\
 & & L_6 \ar@{-}[ru]& \\
 & L_5 \ar@{-}[ru]& & \ar@{-}[lu] L_3\\
 L_4 \ar@{-}[ru]& & L_2 \ar@{-}[lu] \ar@{-}[ru] & \\
 & L_1 \ar@{-}[lu]\ar@{-}[ru] & & \\
K \ar@{-}[ru] & & & &
}  
\]

For any $(M,x) \supseteq (L,w) \supseteq (K,v)$ extension of valued fields where $M/K$ is normal, we have
$L_1=K_{\hv,x} \cap L$, $L_2=K_{\iv,x} \cap L$ and $L_3=K_{\vv,x} \cap
L$, $L_4=L_{K,\sep} \fod K_{\hv,x}$,
$L_5=L_{K,\sep} \fod K_{\iv,x}$ and $L_6=L_{K,\sep} \fod K_{\vv,x}$.

If there is a normal extension $M/K$ containing $L$ such that $\gv_{M,w}=1$, then $L_1=L_4$, $L_2=L_5$ and $L_3=L_6$.
\end{theorem}
 
The proof of the following corollary can be found on Page \pageref{1cp9}.
 
\begin{corollary} \label{1c14324}
Let $(K,v)$ be a valued field and let $L$ and $L'$ be two algebraic extensions contained in an algebraic extension $L''$ of $K$. Let $x$ be a valuation
on
$L''$
extending $v$ and let $w=x|_L$ and $w'=x|_{L'}$. Assume that $K=L \cap L'$ and there exists a normal extension $M/K$ containing $LL'$ such that $\gv_{M,w}=1$.
Then the following statements hold:
\begin{enumerate}
 \item if $x/w'$ is local, then $w/v$ is local; 
 \item if $x/w'$ is totally ramified, then $w/v$ is totally ramified;
 \item if $x/w'$ is totally wild, then $w/v$ is totally wild.
\end{enumerate}
\end{corollary}

The proof of the following proposition can be found on Page \pageref{1cp10}.	

\begin{proposition} \label{1c4343}
Let $(K,v)$ be a valued field and let $L$ be a finite separable algebraic extension of $K$. Let $(M,x) \supseteq (K,v)$ be a finite normal
extension of valued fields with group $G=\Aut_K(M)$ such that the $G$-set $X=\Hom_K(L,M)$ is not empty.
Then the map 
\begin{align*}
 \varphi \colon \Dv_{x,K} \backslash X &\to \{w \textrm{ of }L\textrm{ extending }v\} \\
\Dv_{x,K} s &\mapsto w \textrm{ s.t. } \mathcal{O}_w=\sigma^{-1} ( \mathcal{O}_x \cap \sigma(L))
\end{align*}
is a bijection of sets. If $\varphi(\Dv_{x,K} s)=w$ we have:
\begin{enumerate}
 \item $\# \Dv_{x,K}s = \dv_{\wv}(w/v)\ev_{\tv}(w/v) \fv_{\sv}(w/v)= \nv(w/v)$;
 \item the number of orbits under $\Iv_{x,K}$ of $\Dv_{x,K} s$ is equal to $\fv_{\sv}(w/v)$ and each orbit has length $\dv_{\wv}(w/v)\ev_{\tv}(w/v)$;
 \item the number of orbits under $\Vv_{x,K}$ of $\Dv_{x,K} s$ is equal to $\ev_{\tv}(w/v) \fv_{\sv}(w/v)$ and each orbit has length $\dv_{\wv}(w/v)$.
\end{enumerate}
\end{proposition}

The proof of the following corollary can be found on Page \pageref{1cp11}

\begin{corollary} \label{1c388}
Let $(K,v)$ be a valued field and let $L$ be a finite algebraic extension of $K$. Let $(M,x) \supseteq (K,v)$ be a finite normal extension of valued fields with group $G=\Aut_K(M)$ such that the $G$-set $X=\Hom_K(L,M)$ is not empty. Then the cardinality of the set of valuations $w$ on
$L$ extending $v$ such that $\fv_{\sv}(w/v)=1$ is equal to $\# \left(\Iv_{x,K} \backslash X \right)^{\Dv_{x,K}/\Iv_{x,K}}$.
\end{corollary}

\section{Preliminaries}

\subsection{Field theory}

\subsubsection{Linearly disjoint extensions} \label{1c770}

Let $\Omega$ be a field and let $L,L' \subseteq \Omega$. We set $LL'=\im\left(L
\otimes_{\Z} L' \to \Omega\right)$\index{$LL'$}, that is, the smallest ring containing both $L$ and $L'$. This is a field if the elements of $L$ are
algebraic over $L'$ or if the elements of $L'$ are algebraic over $L$.

Two algebraic field extensions $L, L'$ of a field $K$ are called \emph{linearly disjoint}\index{linearly disjoint}\index{field extension!linearly
disjoint} over $K$ if $L \otimes_K L'$ is a field. This holds if and only if all pairs of finite subextensions of $L/K$ respectively $L'/K$ are
linearly disjoint over $K$.

Let $M/K$ is a normal field extension with group $G=\Aut_K(M)$. Then the latter group is a topological group with the topology coming from viewing $G
\subset M^M$ where $M$ has the discrete topology and $M^M$ the product topology.

\begin{lemma} \label{1c633}
 Let $M/K$ be a normal extension of fields with group $G=\Aut_K(M)$ and let $L, L'$ be two intermediate extensions. Put $H=\Aut_L(M)$ and
$H'=\Aut_{L'}(M)$. Then one has: $\overline{\langle H, H' \rangle}=G$ iff $L \cap L'$ over $K$ is purely inseparable.
\end{lemma}
\begin{proof}
Set $p=\chart(K)$ if $\chart(K)$ is positive and $1$ otherwise. It is very easy to see that $M_{L,\ins} \cap M_{L',\ins}=M_{L \cap L',\ins}$. Note
that
$H=\Aut_{M_{L,\ins}}(M)$, $H'=\Aut_{M_{L',\ins}}(M)$ and $G=\Aut_{M_{K,\ins}}(M)$ and that $M$ is Galois over $M_{K,\ins}$, $M_{L,\ins}$ and
$M_{L',\ins}$ (Proposition \ref{1c777}). From Galois theory it follows that $\overline{\langle H, H' \rangle}$ corresponds to $M_{L,\ins}
\cap M_{L',\ins}=M_{L \cap L',\ins}$ and that $G$ corresponds to $M_{K,\ins}$. Hence one has: $\overline{\langle H, H' \rangle}=G$ iff $M_{L \cap
L',\ins}=M_{K,\ins}$ iff $L \cap L'/K$ is purely inseparable.
\end{proof}

For a field $K$ we denote by $K_{\sep}$ its separable closure\index{$K_{\sep}$}.

\begin{proposition} \label{1c931}
 Let $M/K$ be a normal extension of fields with group $G=\Aut_K(M)$ and let $L, L'$ be two intermediate extensions. Put $H=\Aut_L(M)$ and
$H'=\Aut_{L'}(M)$. Assume that $L/K$ is separable.
Then the following statements are equivalent:
\begin{enumerate}
 \item $L$ and $L'$ are linearly disjoint over $K$;
 \item $L \otimes_K L'$ is a domain;
 \item the natural map $L \otimes_K L' \to LL'$ is an isomorphism;
 \item $G=H \cdot H'$;
 \item $H'$ acts transitively on $G/H$;
 \item the natural map $\Hom_{L'}(LL',K_{\sep}) \to \Hom_K(L,K_{\sep})$ is a bijection.
\end{enumerate}
If $L/K$ or $L'/K$ is normal, then the above statements are equivalent to $L \cap L'=K$.
\end{proposition}
\begin{proof}
i $\iff$ ii: One implication is obvious. Suppose that $L \otimes_K L'$ is a domain. To show that every element has an inverse, we may reduce to the
case where both $L/K$ and $L'/K$ are finite. The result follows since a domain which is finite over a field is a field.

i $\iff$ iii: Obvious.

iv $\iff$ v: Obvious.

v $\iff$ vi: The map in vi is the natural injective map $H'/(H \cap H') \to G/H$. It is surjective iff $H'$ acts transitively on $G/H$.

i $\implies$ vi: The natural map $\Hom_{L'}(LL',K_{\sep}) \to \Hom_K(L,K_{\sep})$ is injective. Let $\varphi \in \Hom_K(L,K_s)$ be given. Let $L''$ be
a finite extension of $K$ contained in $L$. Since $L$ and $L'$ are disjoint over $K$, we find $[LL':L']=[L:K]$. This shows, since $L/K$ is
separable, that the natural injective map $\Hom_{L'}(L'' L',K_{\sep}) \to \Hom_K(L'',K_{\sep})$ is a bijection.
Hence there is a unique morphism of fields in $\Hom_{L'}(L'' L',K_s)$ mapping to $\varphi|_{L''}$. By uniqueness we can glue these morphisms to a unique
morphism mapping to $\varphi$.

iv $\implies$ i: If $G=H \cdot H'$, then for any finite subextension of $L/K$ the same holds. Hence all finite extensions of $L/K$ are linearly
disjoint from $L'$. But then it easily follows that $L$ and $L'$ are linearly disjoint over $K$.

We will now prove the last part. If $L \otimes_K L'$ is a field, then obviously we have $L \cap L'=K$. 
For the other implication, assume first that $L/K$ is normal. This means that $H=\ker(\Aut_K(M) \to \Aut_K(L))$ is a normal subgroup of $G$. But then
one easily sees that $H \cdot H'=\langle H, H' \rangle$. A similar statement holds if $L'/K$ is normal. Furthermore, as $H$ and $H'$ are compact
groups, one sees that $H \cdot H'$ is closed. Hence $\overline{ \langle H, H' \rangle}=H \cdot
H'$. From \ref{1c633}, as $L/K$ is separable, it follows that $H \cdot H' = \overline{ \langle H, H' \rangle}=G$. The result follows.
\end{proof}

\begin{proof}[Proof of Proposition \ref{1c432}] \label{1cp2}
 Let $\mathfrak{L}$ be the set of subfields $F$ of $L$ such that the natural map $L \otimes_F L'F \to LL'$ is an isomorphism. Consider the notation
from i. Directly from the definitions it follows that for a subfield $F$ of $L$ we
have $F \in \mathfrak{L}$ iff $\mathfrak{B}$ spans $L'F$ as $F$-vector space. But $L'F$ is generated as an $F$-vector space by $L'$ and
each $x \in L'$ can be written in a unique way as $x=\sum_{b \in \mathfrak{B}} c_{x,b}b$ where $c_{x,b} \in L$ and almost all $c_{x,b}$ are $0$.
Let $\F$ be the primefield of $L$. Hence we conclude
that $F \in \mathfrak{L}$ iff for all $x \in L$ and $b \in \mathfrak{B}$ we have $c_{x,b} \in F$ iff $F$ contains $M=\F(c_{x,b}: x \in L, b
\in \mathfrak{B})$. 
Description ii follows directly from description one since we can extend an independent set to a basis.
\end{proof}

\begin{definition}
 The field $M$ in the above theorem is called the \index{field of definition}\emph{field of definition} of $L'$ over $L$ and is denoted by $L \fod
L'$\index{$\fod$}. 
\end{definition}

We deduce some properties of $L \fod L'$.

\begin{lemma}
 Let $\Omega$ be a field and let $L,L' \subseteq \Omega$ be subfields. Then the following hold:
\begin{enumerate}
 \item $L \cap L' \subseteq L \fod L'$; 
 \item $L \cap L' = L \fod L'$ iff $L \cap L'= L' \fod L$ iff $L \fod L'=L' \fod L$.
\end{enumerate}
\end{lemma}
\begin{proof}
 i: Suppose $x \in L \cap L' \setminus L \fod L'$. Then the nonzero element $x \otimes 1 - 1 \otimes x$ maps to zero under $L \otimes_{L \fod L'} (L
\fod L')L' \to L L'$, contradiction.

 ii: By symmetry, it suffices to show that the first and last statement are equivalent. Suppose that $L \cap L' = L \fod L'$. Then one has an
isomorphism $L \otimes_{L \cap L'} (L \cap L')L' \to L L'$ and from i one deduces that $L' \fod L=L' \cap L=L \fod L'$.
Suppose $L \fod L'=L' \fod L$. Then one has $L \fod L' \subseteq L \cap L'$ and the result follows from i.  
\end{proof}

\begin{lemma} \label{1c8883}
Let $G$ be a group and let $H, H' \subseteq G$ be subgroups. Let $J'$ be a subgroup of $HH'$ containing $H$. Then $H$ acts transitively on $J'/J' \cap H'$ if and only if $J'$ is contained in the group $\{g \in G: gHH'=HH'\}$.
\end{lemma}
\begin{proof}
First notice $J'/J' \cap H' \cong J'(H'/H') \subseteq G/H'$ (as $J'$-sets). Put $x=H'/H'$. Hence we need to find the largest $J'$ such that $H$ acts transitively on $J'x$, that is
$Hx=J'x$. Notice that $J=\{g \in G: gHH'=HH'\}=\{g \in G: gHx=Hx\}$. If $H$ acts transitively on $J'x$, we have for $j' \in
J'$:
\begin{align*}
 j' Hx=j'J'x=J'x=Hx,
\end{align*}
 hence $j' \in J$. Conversely, $J$ is a subgroup containing $H$ with the property that $Jx=JHx=Hx$.
\end{proof}

\begin{proposition}
 Let $L, L'$ be subfields of a field $\Omega$. Assume that $L/L \cap L'$ is separable. Let $M$ be a normal extension of $L \cap L'$ containing $LL'$ with groups $G=\Aut_{L \cap L'}(M)$, $H=\Aut_{L}(M)$ and $H'=\Aut_{L'}(M)$. Let $J=\{g \in G: gHH'=HH'\}$. Then one has:
\begin{align*}
 L \fod L' = (LL')^J \cap L.
\end{align*}
\end{proposition}
\begin{proof}
Proposition \ref{1c931} shows that we need to find a maximal subgroup $J' \subseteq HH'$ containing $H$ such that $H$ acts 
transitively on $J'/J' \cap H'$. The unique maximal subgroup with this property is $J$ (Lemma \ref{1c8883}). 
It remains to show that $J$ is a closed subgroup. Notice that $H$ and $H'$ are compact, and hence that $HH'$ is compact (because it is the image of
$H \times H'$ under the map $G \times G \to G$) and since we are in a Hausdorff space, it is closed. Similarly, $H'H$ is compact and hence closed.
Note that the translation maps are continous. One then has
\begin{align*}
 J= \bigcap_{\tau \in HH'} \left( \tau H'H \cap HH' \tau^{-1} \right).
\end{align*}
Hence $J$ is an intersection of closed subgroups, and hence closed.
\end{proof}

\subsubsection{Separably disjoint extensions}

Let $L/K$ be an algebraic extension of fields and let $p$ be the characteristic of $K$ if this is nonzero, and $1$ otherwise.
Then we put \index{$L_{K,\ins}$}
\begin{align*}
L_{K,\ins}=\{x \in L: \exists j \in \Z_{\geq 0}: x^{p^j} \in K\},
\end{align*}
the maximal purely inseparable field extension of
$K$ in $L$. Notice that $L_{K,\ins} \cap L_{K,\sep}=K$. 

\begin{definition}
 An algebraic field extension $L/K$ is called \emph{separably disjoint}\index{separably disjoint}\index{field extension!separably disjoint} if
$L=L_{K,\sep}L_{K,\ins}$. 
\end{definition}

\begin{lemma} \label{1c902010}
 Let $L/K$ be an algebraic extension of valued fields. Then $L/K$ is separably disjoint if and only if $L/L_{K,\ins}$ is separable.
\end{lemma}
\begin{proof}
 $\implies$: Follows directly from the definitions.

$\limplies$: Note that $L/L_{K,\sep}$ is purely inseparable and hence $L/L_{K,\sep}L_{K,\ins}$ is purely inseparable and separable. It follows that
$L=L_{K,\sep}L_{K,\ins}$.
\end{proof}

\begin{proposition} \label{1c777}
Let $L/K$ be a normal extension of fields. Then $L/K$ is separably disjoint. 
\end{proposition}
\begin{proof}
See \cite[Chapter V, Proposition 6.11]{LA}.

Here is a similar proof. Take $x \in L \setminus L_{K,\ins}$. As $x$ is not purely inseparable over
$L_{K,\ins}$ and as $L/K$ is normal, there is an element of $\Aut_{K}(L)$ which does not fix $x$ (use Zorn to find such a morphism).
Hence $L^{\Aut_{K}(L)}=L_{K,\ins}$ and from Galois theory it follows that $L/L_{K,\ins}$ is separable. Apply Lemma \ref{1c902010}.
\end{proof}

Notice that any algebraic field extension $L/K$ has a unique maximal separably disjoint subextension, namely $L_{K,\sep}L_{K,\ins}$.

\begin{proposition} \label{1c1111}
 Let $L/K$ be an algebraic extension of fields. Then
\begin{align*}
\varphi \colon \{E: K \subseteq E \subseteq L\} &\to \{(D,F): K \subseteq D \subseteq L_{K,\sep} \subseteq F \subseteq L,\ F/D \textrm{ sep. disj.} \} \\
E &\mapsto (E_{K,\sep}, E L_{K,\sep})
\end{align*}
is a bijection with inverse
\begin{align*}
 (D,F) \mapsto F_{D,\ins}.
\end{align*}
\end{proposition}
\begin{proof}
First we show that $\varphi$ is well-defined. Notice that $E/E_{K,\sep}$ is purely inseparable and that $L_{K,\sep}/E_{K,\sep}$ is separable. Hence
we find that $E L_{K,\sep}/E_{K,\sep}$ is separably disjoint.

Let $\psi$ be the proposed inverse as above. We have $\psi(\varphi(E))=(EL_{K,\sep})_{E_{K,\sep},\ins}$, and this is equal to $E$ since
it obviously contains $E$ and $EL_{K,\sep}/E$ is separable. Conversely we have $\varphi(\psi((D,F)))=((F_{D,\ins})_{K,\sep}, F_{D,\ins}
L_{K,\sep})$. One directly finds $(F_{D,\ins})_{K,\sep}=D$. As $F/D$ is separably disjoint, we find $F_{D,\ins} L_{K,\sep})=F$. This shows that both
maps are inverse to each other.
\end{proof}

\subsection{Tate's lemma}

Let $G$ be a compact topological group which acts continuously on a commutative ring $A$ which is endowed with the discrete
topology. This means that
the map $G \times A \to A$ is continuous. For $a \in A$ the map $G \times \{a\} \to A$ is continuous and the image is compact and hence finite. This
shows that all orbits are finite.

\begin{proposition}[Tate] \label{1cta}\index{Tate's lemma}
Let $(G,A)$ be as above. Let $R$ be a domain and let $\sigma, \tau \colon A \to R$ be ring morphisms. Suppose that $\sigma|_{A^G}=\tau|_{A^G}$. Then there
exists $g \in G$
such that $\tau=\sigma \circ g$.  
\end{proposition}
\begin{proof}
Let $E \subseteq A$ be a finite set. Let $f_E \in A[Y]$ be a polynomial such that all elements of $E$ occur as coefficients of $f_E$.
Extend the action of $G$ to $A[Y][X]$ by letting $G$ act on the coefficients. We extend $\sigma,\tau \colon
A[Y][X] \to R[Y][X]$ by $X
\mapsto X$, $Y \mapsto Y$. Then consider the polynomial $h_E=\prod_{h' \in Gf_E}(X-h')
\in A^G[Y][X]$. We have
\begin{align*}
\prod_{h' \in Gf_E}(X-\sigma(h'))=\sigma(h_E)=\tau(h_E)=\prod_{h' \in Gf_E}(X-\tau(h')).
\end{align*}
As $R[Y]$ is a domain, we can compare the roots and conclude that there is $g \in G$ such that $\tau(h_E)=\sigma(g(h_E)) \in R[Y][X]$. Hence for this
$g$
we have $\tau|_E=\sigma \circ g|_E$.

For any $E \subseteq A$ put $G_E=\{g \in G: \tau|_E=\sigma \circ g|_E\}$. Notice that $G_{\bigcup_{i} E_i}=\bigcap_{i} G_{E_i}$ for any collection of
subsets
$E_i \subseteq A$.
For finite $E$ we have shown $G_E \neq \emptyset$. We claim that for finite $E$ the set $G_E$ is closed in $G$. One
easily shows that for $e \in
E$ the map $\psi_e \colon G \to R$ given by $\psi_e(g)=\sigma(e)-\tau(g(e))$ is continuous. Hence $\psi_e^{-1}(0)=G_{\{e\}}$ is closed. As
$G_E=\bigcap_{e \in E} G_{\{e\}}$, the set $G_E$ is closed.

By compactness of $G$ we have
$G_A=\bigcap_{E \subseteq A,\ E \textrm{ finite}} G_E \neq \emptyset$. This means that there is $g \in G$ such that $\tau=\sigma \circ g$. 
\end{proof}

\begin{corollary} \label{1c312}
Suppose that $(G,A)$ is as above. Let $\pa \subset A^G$ be prime. Then $G$ acts transitively on the set of primes of $A$ lying above $\pa$. 
\end{corollary}
\begin{proof}
Let $\qa, \qa' \subset A$ be primes lying above $\pa$. We will now construct two maps from $A$ to $\overline{Q(A^G/\pa)}$, the algebraic closure of
$Q(A^G/\pa)$. Since the orbits of the actions are finite, the extension $Q(A/\qa) \supseteq Q(A^G/\pa)$ is algebraic. Hence there is a morphism 
$\sigma \colon A \to A/\qa \to
Q(A/\qa) \to \overline{Q(A^G/\pa)}$ which is the identity on $A^G/\pa$.
Similarly one defines another map $\tau \colon A \to A/\qa' \to
Q(A/\qa') \to \overline{Q(A^G/\pa)}$. Both maps agree on $A^G$. Proposition \ref{1cta} says that there is $g \in G$ such that $\tau=\sigma g$. Taking
kernels
gives
$\qa'=\ker \tau=\ker(\sigma g)=g^{-1}\left(\ker \sigma \right)=g^{-1}\qa$. We get $g \qa' = \qa$ and this finishes the proof.  
\end{proof}

\begin{corollary} \label{1c313}
Let $(G,A)$ be as above. Let $\qa \subset A$ be a prime lying above a prime $\pa \subset A^G$. Let
$G_{\qa/\pa}=\{g \in G: g(\qa)=\qa \}$. Let $l=Q(A/\qa)$ and let $k=Q(A^G/\pa)$. Then the natural map $G_{\qa/\pa} \to \Aut_k(l)$ is
surjective and $l/k$ is normal algebraic. 
\end{corollary}
\begin{proof}
It is easy to see that $l/k$ is algebraic. Let $\overline{k}$ be an algebraic closure of $k$ containing $l$. We have a natural map $G_{\qa/\pa} \to \Aut_k(l) \subseteq \Hom_k(l,\overline{k})$. Let $\varphi \in \Hom_k(l,\overline{k})$. Consider the natural map $\sigma \colon A \to Q(A/\qa)=l \subseteq \overline{k}$, which restricts to the natural map $A^G \to Q(A^G/\pa)=k$. Let
$\tau=\varphi \sigma$. Apply Proposition \ref{1cta} to see that there is $g \in G$ with $\varphi \sigma=\sigma g$. But then for $a \in A$ we have
\begin{align*}
g \circ (\sigma(a))= \sigma (g(a))=\varphi \sigma(a).
\end{align*}
This means that $g$ maps to $\sigma$. It follows that $\Aut_k(l)=\Hom_k(l,\overline{k})$ and hence $l/k$ is normal.
\end{proof}

\subsection{Ordered abelian groups}

\begin{lemma} \label{1c60}
Let $(P,\leq)$ be an ordered abelian group. Let $n \in \Z_{\geq 1}$ and $x,y \in P$. If $nx=ny$, then one has $x=y$. The group $P$ has no
non-trivial
torsion and $P \otimes_{\Z} \Q$ is an ordered abelian group where we put $x \leq y$ if for all $n \in \Z_{\geq 1}$ such that $nx, ny \in
P$ we have $nx \leq ny$. 
\end{lemma}
\begin{proof}
Suppose that $x<y$. Then $x+x <x+y<y+y$, and in a similar fashion, $nx<ny$, which is a contradiction. 

If $x$ is torsion, apply the first part to $x$ and $0$ to obtain the second result.  

The last part is an easy calculation which is left to the reader.
\end{proof}

Let $(P,\leq)$ and $(Q,\leq)$ be ordered abelian groups. A morphism $\varphi \colon P \to Q$ is a group homomorphism respecting the ordering. One easily
sees that respecting the order is equivalent to $p \geq 0$ $\implies$ $\varphi(p) \geq 0$. Indeed, let $p, p' \in P$ with $p \geq p'$. Then we have
$p-p' \geq 0$, which gives $\varphi(p)-\varphi(p')=\varphi(p-p') \geq 0$. This gives
$\varphi(p) \geq \varphi(p')$.

\begin{lemma} \label{1c466}
 Let $(P,\leq)$ be an ordered abelian group and let $\varphi \in \Aut(P)$ such that all orbits are finite. Then $\varphi$ is the identity.
\end{lemma}
\begin{proof}
Let $p \in P$ and assume that $\varphi^n(p)=p$. Then one has $p=\varphi^n(p) \geq \ldots \geq \varphi(p)
\geq p$. Hence we obtain $\varphi(p)=p$. 
\end{proof}

\section{Extending valuations}

\begin{lemma} \label{1c944}
 Let $(K,v)$ be a valued field. Then $\mathcal{O}_v$ is integrally closed.
\end{lemma}
\begin{proof}
Suppose $x \in \mathcal{O}_v$ nonzero is integral over $\mathcal{O}_v$. Then there is a relation $x^n+a_{n-1}x^{n-1}+\ldots+a_0=0$ with $a_i \in \mathcal{O}_v$ and this shows that $x \in \mathcal{O}_v[x] \cap \mathcal{O}_v[x^{-1}]$. By the definition of a valuation ring we have $\mathcal{O}_v[x] \cap \mathcal{O}_v[x^{-1}]=\mathcal{O}_v$ and the result follows.
\end{proof}

\begin{proposition} \label{1c143}
Let $K$ be a field. Let $R \subseteq K$ be a subring and let $\pa \in \Spec(R)$. Let $S=\{(A,I): R \subseteq A \subseteq K,\ A
\textrm{ ring},\ I \subseteq A \textrm{\ ideal},\ I \cap R=\pa\}$, ordered by $(A,I) \leq (B,J)$ if $A \subseteq B$ and $I \subseteq J$. Then a pair
$(\mathcal{O},\m)$ is maximal if and only if $\mathcal{O}$ is a valuation ring of $K$ and $\m$ is its maximal ideal.
\end{proposition}
\begin{proof}
Let $(\mathcal{O},\m)$ be a maximal element of $S$. Notice that $\m_{\pa} \subset \mathcal{O}_{\pa}$ satisfies $\m_{\pa} \cap R_{\pa}=\pa R_{\pa}$
and $\m_{\pa} \cap R=\pa$. Hence by maximality we have $\mathcal{O}=\mathcal{O}_{\pa}$ and $\m=\m_{\pa}$. A maximal ideal of $\mathcal{O}$ containing
$\m$ still lies above the maximal ideal of $R_{\pa}$. We conclude that $\m$ is maximal.

We claim that $\mathcal{O}$ is a valuation
ring. Suppose that there is $x \in K^*$ with $x, x^{-1} \not \in
\mathcal{O}$. From the maximality and the fact that $\m$ lies above $\pa R_{\pa}$ one obtains $\m \mathcal{O}[x]=\mathcal{O}[x]$ and $\m
\mathcal{O}[x^{-1}]=\mathcal{O}[x^{-1}]$. Take $n, m$ minimal
such that $1=\sum_{i=0}^n a_i x^i$, $1=\sum_{i=0}^m b_i x^{-i}$ with $a_i, b_i \in \m$. Without loss of generality, assume $m \leq n$. Multiply the
second equation by $x^n$, and notice that $1-b_0 \in \mathcal{O}^*$, to obtain $x^n=\frac{1}{1-b_0} \sum_{i=1}^m b_i x^{n-i}$. Use this relation
together with the first relation to see that $n$ is not minimal, contradiction.

Conversely, suppose that $\mathcal{O}$ is a valuation ring of $K$ with maximal ideal $\m$, containing $R$ and satisfying $\m \cap
R=\pa$. Suppose that $(\mathcal{O},\m)
\leq (A,\na)$. Let $x \in A$ nonzero. Then $x x^{-1}=1 \not \in \na$ and hence $x^{-1} \not \in \m$. As $\mathcal{O}$ is a valuation ring, we obtain
$x \in \mathcal{O}$. Hence $(\mathcal{O},\m)$ is maximal.  
\end{proof}

Since we assume the Axiom of Choice, maximal elements as in Proposition \ref{1c143} exist.

\begin{corollary} \label{1c177}
 Let $R \subseteq L$ be a subring where $L$ is a field. Then the intersection of all valuation rings of $L$ containing $R$ in $L$ is the integral
closure of
$R$ in $L$.
\end{corollary}
\begin{proof}
As a valuation ring is integrally closed (Lemma \ref{1c944}), the right hand side is contained in the left hand side. Suppose
$x \in L$ is not integral over $R$. Consider the
ring
$R[x^{-1}]$, which does not contain $x$ as $x$ is not integral. Hence $x^{-1}$ is
contained in a maximal ideal $\m \subset R[x^{-1}]$. Proposition \ref{1c143} gives us a valuation $v$ with $x^{-1} \in \m_v \cap R[x^{-1}]=\m$. This
is equivalent to $x \not \in \mathcal{O}_v$.
\end{proof}

\begin{proposition} \label{1c434}
 Let $(K,v)$ be a valued field and let $L/K$ be an algebraic extension of fields. Let $R$ be the integral closure of
$\mathcal{O}_v$ in $L$. Then there is a bijection between the set of maximal ideals of $R$ and the set of valuations extending
$v$ to $L$, given by $\m \mapsto R_{\m}$. The inverse maps a valuation $\mathcal{O}$ with maximal ideal $\m$ to $\m \cap R$. 
\end{proposition}
\begin{proof}
Let $\pa \in \MaxSpec(R)$. Then by Proposition \ref{1c143} there exists a valuation ring $\mathcal{O}_w$ of $L$ with
$\mathcal{O}_w \supseteq R_{\pa}$ and $\m_w \cap R=\pa$. We will show $R_{\pa}=\mathcal{O}_w$. 

Let $a \in \mathcal{O}_w$ nonzero. As $L/K$ is algebraic, there exists is a polynomial $f=\sum_{i=0}^n a_i x^i \in \mathcal{O}_v[x]$ with
$f(a)=0$ and a coefficient which is not in the maximal
ideal. Let $k$ minimal such that $a_{k+1},\ldots,a_{n} \in \m_v$. Put $f_0=a_0+\ldots+a_{k-1}x^{k-1}$ and $-f_1=a_k+\ldots+a_nx^{n-k}$. Note that
$f_1(a) \in \mathcal{O}_w^*$. Then from $0=f(a)=f_0(a)-a^kf_1(a)$ we obtain for $b=f_0(a)a^{-k+1}
\in \mathcal{O}_v[a^{-1}]$, $c=f_1(a) \in \mathcal{O}_v[a] \setminus \{0\}$ that $a=\frac{b}{c}$.  We claim: $b, c \in R$. It is
enough to show that $b$ and $c$ are contained in any valuation ring extending $R$ (Corollary \ref{1c177}). Let $\mathcal{O}$ be such a valuation ring.
If $a \in
\mathcal{O}$, then one has $c \in \mathcal{O}$ and hence $b=ac \in \mathcal{O}$. If $a \not \in
\mathcal{O}$, then one has $a^{-1} \in \mathcal{O}$. Hence $b \in \mathcal{O}$ and $c=b a^{-1} \in \mathcal{O}$. This finishes the proof of the
claim. Furthermore, by construction we have $c \not \in \m_w$. Hence $c \not \in \m_w \cap R=\pa$. We see that $a=\frac{b}{c}\in
R_{\pa}$. This gives $R_{\pa}=\mathcal{O}_w$ and this shows that the proposed map is well-defined.

Suppose $w$ extends $v$ to $L$. We want to show that $\m_w \cap R$ is a maximal ideal of $R$. But $\m_w \cap \mathcal{O}_v$ is maximal, and
$\mathcal{O}_v \to R$
is integral. Hence by \cite[Corollary 5.8]{AT} $\m_w \cap R$ is a maximal ideal of $R$. This shows that the proposed inverse is well-defined.

Note that for $\pa \in \MaxSpec(R)$ we have $\pa = \pa R_{\pa} \cap R$. Furthermore, we have already seen $R_{\m_w \cap R}=\mathcal{O}_w$.
This shows that both maps are inverse to each other.
\end{proof}

We will now prove a weak approximation theorem.

\begin{corollary} \label{1c63}
 Let $(K,v)$ be a field and let $L/K$ be an algebraic field extension. Let $w_1,\ldots,w_n$ be different extensions of $v$ to $L$. Let $(a_i)_{i=1}^n \in
\prod_{i=1}^n \mathcal{O}_{w_i}$ and $r_1,\ldots,r_n \in \Z_{\geq 1}$ be given. Then there exists $a \in L$ with $a-a_i \in \m_{w_i}^{r_i}$ for
$i=1,\ldots,n$.
\end{corollary}
\begin{proof}
Let $R$ be the integral closure of $\mathcal{O}_v$ in $L$. Proposition \ref{1c434} gives us maximal ideals $\m_i \in \MaxSpec(R)$ with
$R_{\m_i}=\mathcal{O}_{w_i}$. Using the Chinese remainder theorem, one obtains a surjective map $R \to \prod_{i=1}^n R/\m_i^{r_i}=\prod_{i=1}^n
\mathcal{O}_{w_i}/\m_{w_i}^{r_i}$ and the result follows. 
\end{proof}

\begin{proposition} \label{1c55}
Let $(K,v)$ be a valued field and let $L/K$ be a field extension. Then one has $1 \leq \gv_{L,v}$ and $\gv_{L,v}=1$ if $L/K$ is purely
inseparable. If $(L,x)$ a finite extension of $(K,v)$, then one has $\ev(x/v)\fv(x/v) \leq [L:K]$ and $\gv_{L,v}$ is finite. If the
extension is normal with group
$G=\Aut_K(L)$, then $G$ acts transitively on the set of valuations extending $v$ to $L$, and $\ev(x/v)$ and
$\fv(x/v)$ do not depend on the choice of $x$.
\end{proposition}
\begin{proof}
The fact that $\gv_{L,v} \geq 1$ follows from Proposition \ref{1c143}.

Assume that $L/K$ is purely inseparable. Let $x$ be an extension of $v$ to $L$. Then one directly sees $\m_x=\{r \in L: \exists i: r^{\pv_v^i} \in
\m_v\}$. A valuation is determined by its maximal ideal. 

Assume that $L/K$ is finite. Take a preimage $S \subseteq L$ of a basis of $\kv_x/\kv_v$ and take $T \subseteq L^*$ elements which map bijectively to
$\Delta_x/\Delta_v$. The one easily sees that $ST$ of cardinality $\ev(x/v)\fv(x/v)$ is linearly independent over $K$ and $\ev(x/v)\fv(x/v) \leq
[L:K]$ follows.

Assume that $L/K$ is normal. The transitivity follows from Corollary \ref{1c312} and Proposition \ref{1c434}, and the statements about $\ev(x/v)$ and
$\fv(x/v)$ are obvious. In particular, if $L/K$ is finite normal, the quantity $\gv_{L,v}$ is finite. It follows from Proposition \ref{1c143} that
$\gv_{L,v}$ is finite when $L/K$ is finite.
\end{proof}

\begin{lemma} \label{1c70}
 Let $(K,v)$ be a valued field and let $L$ be a finite normal extension of $K$ of degree $n$. Assume that $x$ is the unique
extension of $v$ to $L$. Then for all $a \in L$ one has $x(a)=\frac{1}{n}v(N_{L/K}(a))$. Furthermore, we have $n \Delta_x \subseteq \Delta_v$.
\end{lemma}
\begin{proof}
Let $G=\Aut(L/K)$. Then it is well-known that for $a \in L$ one has $N_{L/K}(a) = \left(\prod_{g \in G} g(a)\right)^{[L:K]_{\iv}}$. 
We have, keeping in mind Lemma \ref{1c60},
\begin{align*}
 x(a)= \frac{[L:K]_{\iv}}{n} \sum_{g \in G}  x(g(a))  = \frac{1}{n} x(N_{L/K}(a))=\frac{1}{n} v(N_{L/K}(a)).
\end{align*}
The last result follows directly. 
\end{proof}

\begin{lemma} \label{1c72}
Let $(L,w) \supseteq (K,v)$ be a finite purely inseparable extension of valued fields. Then $\kv_{w}/\kv_v$ is purely inseparable and we have
$\ev(w/v)=\ev_{\wv}(w/v)$.
\end{lemma}
\begin{proof}
It is obvious that $\kv_w/\kv_v$ is purely inseparable. Proposition \ref{1c55} together with Lemma \ref{1c70} imply $\ev(w/v)=\ev_{\wv}(w/v)$.
\end{proof}

\section{Normal extensions} 

We will first consider finite extensions of valued fields, and then take a limit.

\subsection{Finite normal extensions}
In this subsection we let $(M,x) \supseteq (K,v)$ be a finite normal extension of valued fields with $G=\Aut_K(M)$. For simplicity, we put
$x_{\iv}=x|_{K_{\iv,x}}$, $x_{\hv}=x|_{K_{\hv,x}}$, $x_{\vv}=x|_{K_{\vv,x}}$ and $x_{\sv}=x|_{M_{K,\sep}}$.

\begin{proposition} \label{1c56}
One has $[K_{\hv,x}:K]=\gv_{M,v}$. Furthermore $x$ is the unique extension of $x_{\hv}$ to $M$ and one has
$\ev(x_{\hv}/v)=\fv(x_{\hv}/v)=1$. 
\end{proposition}
\begin{proof}
Since the action of $G$ on the set of valuations of $M$ extending $v$ is transitive (Proposition \ref{1c55}), we have $[K_{\hv,x}:K]=\gv_{M,v}$. The second statement
also follows from the transitivity of the action. 

We will show $\fv(x_{\hv}/v)=1$. Let $a \in \mathcal{O}_{x_{\hv}}$, and pick $\alpha_a \in K_{\hv,x}$ satisfying $\alpha_a-a \in
\m_{x_{\hv}}$
and
$\alpha_a$ 
in the maximal ideal of any other valuation extending $v$ to $K_{\hv,x}$ (Corollary \ref{1c63}). This means that for $\overline{g} \in G/\Dv_{x,K}$ with $g \neq \Dv_{x,K}$ we have $g(\alpha) \in \m_{x_{\hv}}$. Then, by looking in $M$, one obtains
\begin{align*}
\tr_{K_{\hv,x}/K}(\alpha_a)=\sum_{\overline{g} \in G/\Dv_{x,K}} g(\alpha) \in a+\m_{x_{\hv}}.
\end{align*}
Notice that $\tr_{K_{\hv,x}/K}(\alpha_a) \in K \cap \mathcal{O}_{x_{\hv}}=\mathcal{O}_v$. Hence the natural map $\kv_v \to \kv_{x_{\hv}}$ is surjective. This gives $\fv(x_{\hv}/v)=1$.

Next we will prove $\ev(x_{\hv}/v)=1$. Let $b \in K_{\hv,x}^*$. Take $m \in \Z$ such that for all $g \in G \setminus \Dv_{x,K}$, we have
$x_{\hv}(\alpha_1^m b) \neq x_{\hv}(g(\alpha_1^mb))$. To do this, one needs to make sure that for all $g \in G \setminus \Dv_{x,K}$ one has
$m(x(\alpha_1)-x(g(\alpha_1))) \neq x(g(b))-x(b)$, which can easily be achieved since $x(\alpha_1) \neq
x(g(\alpha_1))$, the group $\Delta_x$ is torsion-free and $G$ is finite. 
Put $\beta=\alpha_1^m b$ and $f=\prod_{\overline{g} \in G/\Dv_{x,K}} (X-\overline{g}(\beta))=\sum_{i=0}^n a_i X^i \in K[X]$ with $a_n=1$. Let $S=\{\overline{g}(\beta): \overline{g} \in G/\Dv_{x,K} \textrm{ s.t. } 
x_{\hv}(\overline{g}(\beta))<x_{\hv}(\beta)\}$ and set $r=\#S$. Then one sees $x_{\hv}(a_{n-r})=x_{\hv}(\prod_{c \in S} c)$
and
$x_{\hv}(a_{n-r-1})=x_{\hv}(\beta \prod_{c \in S} c)$. This gives $x_{\hv}(b)=x_{\hv}(\beta)=x_{\hv}(a_{n-r-1}/a_{n-r}) \in \Delta_v$ and we are done.
\end{proof}

\begin{proposition} \label{1c57}
We have a short exact sequence 
\begin{align*}
0 \to \Iv_{x,K} \to \Dv_{x,K} \to  \Aut_{\kv_v}(\kv_x) \to 0
\end{align*}
and $\kv_x$ is normal over $\kv_v$. Furthermore
we have $[K_{\iv,x}:K_{\hv,x}]=\fv(x_{\iv}/x_{\hv})=\fv_{\sv}(x_{\iv}/x_{\hv})=\fv_{\sv}(x/v)$ and $\ev(x_{\iv}/x_{\hv})=1$.
\end{proposition}
\begin{proof}
The exactness of the sequence and the normality of $\kv_x$ over $\kv_v$ follow from Corollary \ref{1c313} and Proposition \ref{1c434}.
We will now prove the last two statements. Look at the normal
extension $M/K_{\iv,x}$ with group $\Iv_{x,K}$. From the exact sequence for the extension $M/K_{\iv,x}$ just obtained we see that the zero map $\Iv_{x,K} \to \Aut_{\kv_{x_{\iv}}}(\kv_x)$ is surjective. We find $\Aut_{\kv_{x_{\iv}}}(\kv_x)=0$. As $\kv_x/\kv_{x_{\iv}}$ is normal, this gives that $\kv_x/\kv_{x_{\iv}}$ is purely
inseparable.
Consider the Galois extension $K_{\iv,x}/K_{\hv,x}$ with group $\Dv_{x,K}/\Iv_{x,K}$. We obtain an exact sequence 
\begin{align*}
\Dv_{x,K}/\Iv_{x,K}	 \to
\Aut_{\kv_{x_{\hv}}}(\kv_{x_{\iv}})=\Aut_{\kv_v}(\kv_x) \to 0 
\end{align*}
(note that $\kv_v=\kv_{x_{\hv}}$ by Proposition \ref{1c56}). The first map
is injective and hence we have an isomorphism. Using Proposition \ref{1c55} we obtain:
\begin{align*}
\ev(x_{\iv}/x_{\hv}) \cdot \fv(x_{\iv}/x_{\hv}) &\leq [K_{\iv,x}:K_{\hv,x}]=\# \Aut_{\kv_{x_{\hv}}}(\kv_{x_{\iv}})\\
&= \# \Aut_{\kv_{v}}(\kv_{x}) \leq [\kv_{x_{\iv}}:\kv_{x_{\hv}}] =
\fv(x_{\iv}/x_{\hv}).
\end{align*}
Hence we find $[K_{\iv,x}:K_{\hv,x}]=\fv(x_{\iv}/x_{\hv})=\fv_{\sv}(x_{\iv}/x_{\hv})=\fv_{\sv}(x/v)$ and $\ev(x_{\iv}/x_{\hv})=1$.
\end{proof}

\begin{lemma} \label{1crene}
 Let $0 \to A \to B \overset{f}{\to} C \to 0$ be an exact sequence of abelian groups. Let $H$ be the group of automorphisms of this sequence consisting of automorphisms which are the identity on $C$. Let $H' \subseteq H \subseteq \Aut(B)$ the set of automorphisms which are the identity on $A$. Then the map
 \begin{align*}
\varphi \colon H &\to \Hom(C,A) \rtimes \Aut(A) \\
h &\mapsto (f(b) \mapsto h(b)-b, h|_A) 
 \end{align*}
 is an injective morphism of groups. One has:
 \begin{enumerate}
 \item $\varphi|_{H'} \colon H' \to \Hom(C,A)$ is an isomorphism;
 \item if  $0 \to A \to B \to C \to 0$ is split, then $\varphi$ is an isomorphism. 
 \end{enumerate}
\end{lemma}
\begin{proof}
One easily shows that $\varphi$ is well-defined and that it is a morphism of groups.

i: Consider the following map:
\begin{align*}
\psi \colon \Hom(C,A) &\to H' \\
\sigma &\mapsto (b \in B \mapsto b+\sigma(f(b)))
\end{align*}
One then easily checks that this is the inverse of $\varphi|_{H'}$. This also shows that $\varphi$ is injective.

ii: Consider the exact sequence $0 \to A \to A \oplus C \to C \to 0$. Consider the map
\begin{align*}
\chi \colon \Hom(C,A) \rtimes \Aut(A) &\to H \\
(\sigma,\tau) &\mapsto \left((a,c) \mapsto (\sigma(c)+\tau(a),c) \right).
\end{align*}
One easily checks that both maps are inverse to each other.
\end{proof}

We have an exact sequence $1 \to \mathcal{O}_x^*/(1+\m_x)\cong \kv_x^* \to M^*/(1+\m_x) \to \Delta_x \to 1$. 
Note that $\Dv_{x,K}$ acts on such sequences and it acts on this sequence trivially on $\Delta_x$ (Lemma \ref{1c466}), it fixes $K^*/(1+\m_v)$ and the action on $\kv_x^*$ comes from a field automorphism. Set $\Gamma_{x,v}=\im\left(\Aut_{\kv_v}(\kv_x) \to \Aut_{\kv_v^*}(\kv_x^*) \right)$. Let $\Aut_{K^*,\Gamma_{x,v}}(M^*/(1+\m_x))$ be the group of automorphisms of the sequence, seen as subgroup of $\Aut(M^*/(1+\m_x))$, which are the identity on $K^*/(1+\m_x)$ and which induce an element of $\Gamma_{x,v}$ on $\kv_x^*$ (note that the two conditions already imply that they act as the identity on $\Delta_x$). 
We get a morphism $\Dv_{x,K} \to \Aut_{K^*,\Gamma_{x,v}}(M^*/(1+\m_x))$.

Note that the group $\Iv_{x,K}$ acts trivially on $\kv_x^*$ and $K^*/(1+\m_v)$ and $\Delta_x$. The automorphisms of the exact sequence with these properties correspond by Lemma \ref{1crene} to $\Hom(\Delta_x/\Delta_v,\kv_x^*)$ and this gives a morphism \begin{align*}
\Iv_{x,K} &\to \Hom(\Delta_x/\Delta_v,\kv_x^*) \\
g &\mapsto \left(\overline{c} \mapsto
\overline{g(c)/c} \right).
\end{align*}
By definition $\Vv_{x,K}$ is the kernel of the last morphism.

\begin{lemma} \label{1c66}   
 Let $(L',u') \supseteq (L,u)$ be a finite normal extension of valued fields with group $H$. Assume that for all $a \in L'^*$ and $h \in H$ we have
\begin{align*}
 \frac{h(a)}{a} \in 1+\m_{u'}.
\end{align*}
Then $H$ is a $\pv_{u}$-group.
\end{lemma}
\begin{proof}
We can directly reduce to the case where $L'/L$ is Galois and $[L':L]>1$. 
Let $a \in L'^*$ satisfy $\tr_{L'/L}(a)=0$, which exists by looking at
dimensions. Then we have 
\begin{align*}
 0 = \frac{\tr_{L'/L}(a)}{a} \in [L':L]+\m_{u'}.
\end{align*}
This shows $\#H=[L':L]=0 \in \kv_{u}$ and hence $H$ is a $\pv_{u}$-group.
\end{proof}

\begin{proposition} \label{1c58}
 The subgroup $\Vv_{x,K}$ is the unique $\pv_v$-Sylow subgroup of $\Iv_{x,K}$. The sequences 
\begin{align*}
0 \to \Vv_{x,K} \to \Iv_{x,K} \to
\Hom(\Delta_x/\Delta_v,\kv_x^*) \to 0
\end{align*}
and
\begin{align*}
0 \to \Vv_{x,K} \to \Dv_{x,K} \to \Aut_{K^*,\Gamma_{x,v}}(M^*/(1+\m_x)) \to 0
\end{align*} 
are exact. One has $[K_{\vv,x}:K_{\iv,x}]=\ev(x_{\vv}/x_{\iv})=\ev_{\tv}(x_{\vv}/x_{\iv})$ and $\fv(x_{\vv}/x_{\iv})=1$. We also
have $[M:K_{\vv,x}] \in \pv_v^{\Z_{\geq 0}}$ and $\ev(x/x_{\vv})=\ev_{\wv}(x/x_{\vv})$. Set $s=\ord(\Iv_{x,K}/\Vv_{x,K})$. Then $\kv_{x}$ has an $s$-th primitive root of unity.
\end{proposition}
\begin{proof}
Let $\varphi \colon \Iv_{x,K} \to \Hom(\Delta_x/\Delta_v,\kv_x^*)$ be the morphism with kernel $\Vv_{x,K}$. Since $\kv_x^*$ has no non-trivial elements of $\pv_v$-power order, all elements
of order a power of $\pv_v$ of $\Iv_{x,K}$ are in $\Vv_{x,K}$. Consider the normal
extension $M/K_{\vv,x}$ with automorphism group $\Vv_{x,K}$. One obtains that $\Vv_{x,K}$ is a $\pv_v$-group by Lemma \ref{1c66}. Hence $\Vv_{x,K}$ is the unique
$\pv_v$-Sylow subgroup of $\Iv_{x,K}$.

Consider the Galois extension $K_{\vv,x}/K_{\iv,x}$ with group $\Iv_{x,K}/\Vv_{x,K}$. As the order of 
$\Iv_{x,K}/\Vv_{x,K}$ is coprime with $\pv_v$, we have an exact sequence $0 \to \Iv_{x,K}/\Vv_{x,K} \to \Hom(\Delta_{x_{\vv}}/\Delta_{x_{\iv}},\kv_{x_{\vv}}^*)$. Using
Proposition
\ref{1c55} we obtain 
\begin{align*}
\fv(x_{\vv}/x_{\iv}) \cdot \ev(x_{\vv}/x_{\iv})
\leq [K_{\vv,x}:K_{\iv,x}] \leq \#\Hom(\Delta_{x_{\vv}}/\Delta_{x_{\iv}},\kv_{x_{\vv}}^*) \leq \ev_{\tv}(x_{\vv}/x_{\iv}) \leq
\ev(x_{\vv}/x_{\iv}).
\end{align*}
We obtain $\fv(x_{\vv}/x_{\iv})=1$ and $[K_{\vv,x}:K_{\iv,x}]=\ev(x_{\vv}/x_{\iv})=\ev_{\tv}(x_{\vv}/x_{\iv})$. We also obtain $\#\Hom(\Delta_{x_{\vv}}/\Delta_{x_{\iv}},\kv_{x_{\vv}}^*)= \# \left( \Iv_{x,K}/\Vv_{x,K} \right)$. Finally we obtain that $\kv_{x_{\vv}}$ has an $s$-th primitive root of unity.

One easily obtains $[M:K_{\vv,x}] \in \pv_v^{\Z_{\geq 0}}$. Furthermore, $\ev(x/x_{\vv})=\ev_{\wv}(x/x_{\vv})$ follows from Lemma \ref{1c70}.

The extension $\kv_{x}/\kv_{x_{\vv}}$ is purely inseparable (Proposition \ref{1c57}). This shows that the torsion of $\kv_x^*$ is equal to the torsion of $\kv_{x_{\vv}}^*$. Note that $\Delta_v=\Delta_{x_{\iv}}$ by Proposition \ref{1c56} and Proposition \ref{1c57}. Hence we have a natural map
$\Hom(\Delta_x/\Delta_v,\kv_x^*) \to \Hom(\Delta_{x_{\vv}}/\Delta_{x_{\iv}},\kv_{x_{\vv}}^*)$. Because $\Delta_x/\Delta_{x_{\vv}}$ is a $\pv_{v}$-group, this map is injective. We find
\begin{align*}
 \# \left( \Iv_{x,K}/\Vv_{x,K} \right) \leq \# \Hom(\Delta_x/\Delta_v,\kv_x^*) \leq \# \Hom(\Delta_{x_{\vv}}/\Delta_{x_{\iv}},\kv_{x_{\vv}}^*) =  \# \left( \Iv_{x,K}/\Vv_{x,K} \right).
\end{align*}
This shows that the sequence $0 \to \Vv_{x,K} \to \Iv_{x,K} \to \Hom(\Delta_x/\Delta_v,\kv_x^*) \to 0$ is exact.

We will show that the second sequence is exact. Recall that this sequence comes from the action of $\Dv_{x,K}$ on $0 \to \kv_x^* \to M^*/(1+\m_x) \to \Delta_x \to 0$. If $\sigma \in \Dv_{x,K}$ acts trivially on the exact sequence, then it acts trivially on $\kv_x$ and hence lies in $\Iv_{x,K}$ and hence in $\Vv_{x,K}$. We will count $\Aut_{K^*,\Gamma_{x,v}}(M^*/(1+\m_x))$. Note that the restriction map to $\Aut_{K^*,\Gamma_{x,v}}(M^*/(1+\m_x)) \to \Gamma_{x,v}$, a group of cardinality $\#\Dv_{x,K}/\#\Iv_{x,K}$, is surjective (Proposition \ref{1c57}). Let $h  \in \Gamma_{x,v}$. The set of automorphisms inducing $h$ is in bijection with $\Hom(\Delta_x/\Delta_v,\kv_x^*)$ (Lemma \ref{1crene}), and this set is of cardinality $\#\Iv_{x,K}/\#\Vv_{x,K}$. We find:
\begin{align*}
\#\Aut_{K^*,\Gamma_{x,v}}(M^*/(1+\m_x)) = \#\Dv_{x,K}/\#\Iv_{x,K} \cdot \#\Iv_{x,K}/\#\Vv_{x,K}= \#\Dv_{x,K}/\#\Vv_{x,K}.
\end{align*}
Hence the last sequence is exact. 
\end{proof}

We later use the following lemma, which summarizes part of the situation. In Proposition \ref{1c854} we will give a more readable form.

\begin{lemma} \label{1c90}	
Let $(M,x) \supseteq (K,v)$ be a finite normal extension of valued fields. Then the following statements hold:
\begin{enumerate}
 \item $[K_{\hv,x}:K]=\gv_{M,v}$, $\ev(x_{\hv}/v)=\fv(x_{\hv}/v)=1$; 
\item $[K_{\iv,x}:K_{\hv,x}]=\fv_{\sv}(x_{\iv}/x_{\hv})=\fv(x_{\iv}/x_{\hv})=\fv_{\sv}(x/v)$, $\ev(x_{\iv}/x_{\hv})=1$;
 \item $[K_{\vv,x}:K_{\iv,x}]=\ev_{\tv}(x_{\vv}/x_{\iv})=\ev(x_{\vv}/x_{\iv})=\ev_{\tv}(x/v)$, $\fv(x_{\vv}/x_{\iv})=1$;
 \item $[M:K_{\vv,x}] \in \pv_v^{\Z_{\geq 0}}$, $\ev(x/x_{\vv})=\ev_{\wv}(x/x_{\vv})=\ev_{\wv}(x/v)$,
$\fv(x/x_{\vv})=\fv_{\iv}(x/x_{\vv})=\fv_{\iv}(x/v)$.
\end{enumerate}
\end{lemma}
\begin{proof}
This follows from combining Proposition \ref{1c56}, Proposition \ref{1c57}, Proposition \ref{1c58} and Proposition \ref{1c55}.
\end{proof}

\subsection{Normal extensions}

\begin{remark}
 Let $(M,x) \supseteq (K,v)$ be a normal extension of valued fields and let $G=\Aut_K(M)$. Let $\mathfrak{T}$ be the set of
finite normal extensions of $K$ in $M$. By definition one has
\begin{align*}
 \Dv_{x,K} &= \underset{\underset{M' \in \mathfrak{T}}{\longleftarrow}}{\lim} \Dv_{x|_{M'},K}; \\
 \Iv_{x,K} &= \underset{\underset{M' \in \mathfrak{T}}{\longleftarrow}}{\lim} \Iv_{x|_{M'},K}; \\
 \Vv_{x,K} &= \underset{\underset{M' \in \mathfrak{T}}{\longleftarrow}}{\lim} \Vv_{x|_{M'},K}.
\end{align*}
 All maps in the projective limits come from the natural restriction maps. 
\end{remark}

\begin{lemma} \label{1c352}
 Let $K \subseteq L \subseteq M$ be algebraic extensions of fields. Then we have $M_{L,\sep}=M_{K,\sep}L$.
\end{lemma}
\begin{proof}
Assume $p=\chart(K) \neq 0$.
One obviously has $M_{L,\sep} \supseteq M_{K,\sep}L$. For $x \in M \setminus M_{K,\sep}L$ there is $i \in \Z_{\geq 0}$ such that $x^{p^i} \in
M_{K,\sep}$, and hence $M/M_{K,\sep}L$ is purely inseparable and we are done. 
\end{proof}

We need the following technical lemma.

\begin{lemma} \label{1c354}
Let $M \supseteq L \supseteq K$ be algebraic extensions of fields where $M/K$ is normal. Let $G=\Aut_{K}(M)$ and
$H=\Aut_{L}(M) \subseteq
G$. 
Let $G' \subseteq G$ be a subgroup. Then one has $\left(M_{L, \sep}\right)^{G' \cap H}=\left(M_{K,\sep}\right)^{G'}L$. 
\end{lemma}
\begin{proof}
By Proposition \ref{1c1111} it is enough to show that both fields have the same compositum and intersection with $M_{K,\sep}$. 
We start with the intersection, where we use Galois theory (for the first equality, note that
$\left(M_{K,\sep}\right)^{G'}L/\left(M_{K,\sep}\right)^{G'}  (L \cap M_{K,\sep})$ is purely
inseparable):
\begin{align*}
\left(\left(M_{K,\sep}\right)^{G'}L \right)\cap M_{K,\sep}&= \left(M_{K,\sep}\right)^{G'}  (L \cap M_{K,\sep}) \\
&=  \left(M_{K,\sep}\right)^{G'} \left(M_{K,\sep} \right)^H = \left(M_{K,\sep}\right)^{G' \cap H} \\
&= \left(M_{L,\sep}\right)^{G' \cap H} \cap M_{K,\sep}.
\end{align*}
For the compositum, we have:
\begin{align*}
\left( \left( M_{K,\sep} \right)^{G'} L \right) M_{K,\sep} &= M_{K,\sep}L
\end{align*}
and by Lemma \ref{1c352} we have
\begin{align*}
 L M_{K,\sep} \subseteq \left(M_{L,\sep}\right)^{H \cap G'} M_{K,\sep} \subseteq M_{L,\sep}
M_{K,\sep}=LM_{K,\sep}.
\end{align*}
The result follows.
\end{proof}

\begin{proof}[Proof of Theorem \ref{1c333}] \label{1cp1}
We will first prove that the map $G \times S \to S$ is continuous. Take $x' \in S$, $L$ a finite extension of $K$ and suppose that $g \cdot x'' \in
U_{x',L}$. Let $N/K$ be a finite normal extension containing $L$ in $M$. Then one has $\{g' \in G: g|_N=g'|_N\} \cdot U_{x'',N} \subseteq U_{x',L}$.
This shows that the action is continuous. By definition the stabilizer of $x$ is $\Dv_{x,K}$ and this gives us the isomorphism $G/\Dv_{x,K} \to S$. 

One easily obtains for $g \in G$ the equalities $\Dv_{g(x),K}=g \Dv_{x,K} g^{-1}$ and $K_{\hv,g(x)}=g K_{\hv,x}$. The other cases are similar.

We want to show that the sequences are exact. The idea is that the result is a limit of the statements for finite normal extensions. This is the
reason why the maps are morphisms of profinite groups. Let $0 \to A_i \to B_i \to C_i \to 0$ ($i$ is some indexed set) be exact sequences of groups such that we can take a projective limit. Then the remaining sequence is left-exact. It is exact if all the maps $A_i \to A_j$ in the system are surjective (in this
case, the so-called Mittag-Leffler condition is satisfied). See for example \cite[Proposition 10.2]{AT} for a statement which is sufficient. Hence we take the limit of the sequences from Proposition \ref{1c57} and Proposition \ref{1c58}.

Let $M', M''$ be finite normal over $K$ with $K \subseteq M' \subseteq M'' \subseteq M$. We claim that the natural maps $\Iv_{x|_{M''},K}
\to \Iv_{x|_{M'},K}$ and $\Vv_{x|_{M''},K} \to \Vv_{x|_{M'},K}$ are surjective. Take $g \in \Iv_{x|_{M'},K}$ with lift $g' \in \Aut_K(M'')$. We have
an exact sequence $0 \to \Iv_{x|_{M''},M'} \to \Dv_{x|_{M''},M'} \to \Aut_{\kv_{x|_{M'}}}(\kv_{x|_{M''}}) \to 0$ from Proposition \ref{1c57}. This
shows that there is a $g'' \in \Dv_{x|_{M''},M'}  \subseteq \Aut_{M'}(M'')$ such that $g'g'' \in \Iv_{x|_{M''},K}$ and $\overline{g'g''}=g \in
\Aut_K(M')$. A similar proof, using Proposition \ref{1c58} and the result just obtained, shows the surjectivity of $\Vv_{x|_{M''},K}
\to \Vv_{x|_{M'},K}$.

This shows that all sequences in the limit remain exact (for the first one, we could have also used Corollary \ref{1c313}). Since $\kv_{x|_{M'}}$ has enough roots of unity (Proposition \ref{1c58}) we find
\begin{align*}
\Hom(\Delta_x/\Delta_v,\kv_x^*)= \Hom( \underset{\underset{M'}{\rightarrow}}{\lim}{\Delta_{x|_{M'}}/\Delta_v,\kv_x^*)})=\underset{\underset{M'}{\leftarrow}}{\lim} \Hom(\Delta|_{x|_{M'}}/\Delta_v, \kv_{x|_{M'}}^*).
\end{align*}

For the limit of the third sequence, we need to prove
\begin{align*}
 \Aut_{K^*,\Gamma_{x,v}}(M^*/(1+\m_x)) =\underset{\underset{M'}{\leftarrow}}{\lim}  \Aut_{K^*,\Gamma_{x|_{M'},v}}(M^*/(1+\m_{x|_{M'}}))
\end{align*}
It is easy to see that the right group is contained in the left group. The other implication follows since such an automorphism of 
\begin{align*}
1 \to \kv_x^* \to M^*/(1+\m_x) \to \Delta_{x} \to 1
\end{align*}
induces an automorphism of
\begin{align*}
1 \to \kv_{x|_{M'}}^* \to M'^*/(1+\m_{x|_{M'}}) \to \Delta_{x|_{M'}} \to 1
\end{align*}
because $\kv_{x|_{M'}}/\kv_v$ is normal (Proposition \ref{1c57}).

The statement about the pro-$\pv_v$-Sylow statement follow from Proposition \ref{1c58}. The normality of $\kv_x/\kv_v$ follows
from Proposition \ref{1c57}. Proposition \ref{1c58} also gives the statement about the roots of unity.

Statement i directly follows from the definition. Statement ii follows from statement i and Lemma \ref{1c354}.

We will prove the exactness of the last three sequences. The only non-trivial part is the surjectivity of the last maps. The exactness for
the last two sequences is as before, and the exactness of the first sequence follows from the transitivity of the action of $G$ on $S$. The last
statements then follow directly.
\end{proof}

\begin{lemma} \label{1c1320}
Suppose we have the following commutative diagram of groups with exact rows:
\[
\xymatrix{
0 \ar[r] & A \ar[r] \ar[d] & B \ar[r]^{h} \ar[d]^{f} & C \ar[r] \ar@{->>}[d]^{g} & 0 \\
0 \ar[r] & A' \ar[r] & B' \ar[r]^{h'} & C' \ar[r] & 0.
}  
\]
Assume that $s \colon C \to B$ is a splitting of the first exact sequence and assume that $\Hom(\ker(g),A')=0$. Then the map
\begin{align*}
s' \colon C' &\to B' \\
c'=g(c) &\mapsto f \circ s(c)
\end{align*}
is well-defined and it is a splitting of the second exact sequence.
\end{lemma}
\begin{proof}
Consider the morphism $f \circ s|_{\ker(g)} \colon \ker(\varphi) \to B'$. Note that the image actually lands in $A'$ and hence this is the $0$ map. Hence $s'$ is a well-defined map. Take $c'_1, c'_2 \in C'$, say with preimage $c_1$ respectively $c_2$ in $C$. Then $c_1+c_2$ is a preimage of $c'_1+c'_2$ and hence have
\begin{align*}
s'(c'_1+c'_2) = f \circ s(c_1+c_2) = f \circ s(c_1)+f \circ s(c_2)=s'(c'_1)+s'(c'_2).
\end{align*}
This shows that $s'$ is a morphism. For $c' \in C'$ with preimage $c \in C$ we find
\begin{align*}
h'(s'(c'))= h' \circ f \circ s(c) = g \circ h \circ s(c)= g(c)=c'.
\end{align*}
\end{proof}

\begin{proof}[Proof of Theorem \ref{1c7776}] \label{1c7779}

We may assume that $K=K_{\hv,x}$. Set $\Gamma=\Aut_{\kv_v}(\overline{\kv_x})$.

i: Let $\overline{M}$ be an algebraic closure of $M$ with valuation $x'$ extending $x$. Let $S$ be the set of intermediate fields $L$ of $\K_{\vv,x'}/K$ with the property that $(L,x'|_L) \supseteq (K,v)$ is totally ramified and tame. Order this set by inclusion. This set is not empty. Note that a chain has an upper bound, namely the union. By Zorn there is a maximal element, say $L'$. Notice that $\Delta_{x'}/\Delta_{x'|_{L'}}$ is a $\pv_v$-group. Indeed, if not, we could find a totally and tamely ramified extension of $L'$ in $\overline{M}$ by taking a root of an element of $L'$. 
Using the exact sequences (Theorem \ref{1c333}) it is not hard to see that $\Gal(K_{\vv,x'}/L') \cong \Gamma$ (the extension $K_{\vv,x'}/L'$ has trivial $\Vv$ and hence trivial $\Iv$). This shows that the sequence  $0 \to \Iv_{x',K}/\Vv_{x',K} \to \Dv_{x',K}/\Vv_{x',K} \to \Gamma \to 0$ is split.
We have the following commutative diagram:
\[
\xymatrix{
0 \ar[r] & \Iv_{x',K}/\Vv_{x',K} \ar[r] \ar[d] & \Dv_{x',K}/\Vv_{x',K} \ar[r] \ar[d] & \Gamma \ar[r] \ar@{->>}[d] & 0 \\
0 \ar[r] & \Iv_{x,K}/\Vv_{x,K} \ar[r] & \Dv_{x,K}/\Vv_{x,K} \ar[r] & \Aut_{\kv_v}(\kv_x) \ar[r] & 0.
}  
\]
Lemma \ref{1c1320} gives us a splitting of the second sequence provided that we can show $\Hom(\Aut_{\kv_x}(\overline{\kv_x}),\Iv_{x,K}/\Vv_{x,K})=0$. Note that $\Iv_{x,K}/\Vv_{x,K} \cong \Hom(\Delta_x/\Delta_v,\kv_x^*)$ (Theorem \ref{1c333}). Suppose we have such a non-trivial morphism. Then as $\Delta_x/\Delta_v$ is torsion, we can find an element of prime order $l$ coprime to $\pv_v$ in $\Delta_x/\Delta_v$ not mapping to zero (note that the order of $l$ divids $\ord(\Iv_{x,K}/\Vv_{x,K})$). This gives us a surjective morphism $\Aut_{\kv_x}(\overline{\kv_x}) \to \Z/l\Z$. By assumption such a morphism does not exist.

ii: Note that $\Vv_{x,K}$ is the unique pro-$\pv_v$-Sylow subgroup of $\Iv_{x,K}$ (Theorem \ref{1c333}) and hence the statement trivially follows when $\pv_v=1$ or when $\pv_v \nmid \ord(\Iv_{x,K})$. Assume $\pv_v \neq 1$ and $\pv_v | \ord(\Iv_{x,K})$. It is enough to show that the cohomological $\pv_v$-dimension of $\Dv_{x,K}/\Vv_{x,K}$, $\cd_{\pv_v}(\Dv_{x,K}/\Vv_{x,K})$, is at most $1$ \cite[Chapter I, Proposition 16]{SE5}. We have an exact sequence $0 \to \Iv_{x,K}/\Vv_{x,K} \to \Dv_{x,K}/\Vv_{x,K} \to \Dv_{x,K}/\Iv_{x,K} \to 0$. From Theorem \ref{1c333} it follows that $\Dv_{x,K}/\Iv_{x,K}$ is isomorphic to $\Aut_{\kv_v}(\kv_x)$. This group has cohomological dimension at most $1$  (see the proof of \cite[Chapter II, Proposition 3]{SE5} or \cite[Theorem 22.2.1]{EFR}, in combination with Artin-Schreier theory). We have an isomorphism $\Iv_{x,K}/\Vv_{x,K} \cong \Hom(\Delta_x/\Delta_v,\kv_x^*)$ and its order is coprime to $\pv_v$. Hence we have $\cd_{\pv_v}(\Iv_{x,K}/\Vv_{x,K})=0$ (\cite[Chapter I, Corollary 2 on Page 19]{SE5}). Using \cite[Chapter I, Proposition 15]{SE5} we see that $\cd_{\pv_v}(\Dv_{x,K}/\Vv_{x,K}) \leq 1$ and the result follows.

iii: This follows after combining i and ii and the fact that $\Vv_{x,K}$ is a $\pv_v$-group (Theorem \ref{1c333}).
\end{proof}

\section{Algebraic extensions} 

\subsection{Proofs}

\begin{proposition} \label{1c4444}
 Let $(L,w) \supseteq (K,v)$ be a finite extension of valued fields. Let $(M,x)$ be a finite normal extension of $(K,v)$ containing $(L,w)$. Then the
following hold:
\begin{enumerate}
 \item the quantity $\nv(w/v)$ is well-defined and one has $n(w/v)=\frac{\gv_{M,w}}{\gv_{M,v}} \cdot [L:K] = [L_{\hv,x}:K_{\hv,x}]$;
 \item $\dv(w/v)$ is well-defined and has values in $\pv_v^{\Z_{\geq 0}}$. 
\end{enumerate}
Furthermore, the quantities $\dv$, $\dv_{\wv}$, $\ev$, $\ev_{\tv}$ $\ev_{\wv}$, $\fv$, $\fv_{\sv}$, $\fv_{\iv}$ and $\nv$ are
multiplicative in towers.
\end{proposition}
\begin{proof}
 
i. We will show that $\nv(w/v)$ is well-defined, i.e. does not depend on the choice of $M$. Let $M'$ be another normal extension of $K$ containing $L$ with
$G=\Aut_K(M')$. Without loss of generality, we may assume $M' \supseteq M$. 
Put $H=\Aut_M(M')$. Let $X$ (respectively $X'$) be the set of primes of $M$ (respectively $M'$) extending $v$. Note
that $G$ acts transitively on $X'$, and $G/H$ acts transitively on $X$ (Proposition \ref{1c55}). Then one easily shows that the map $X' \to X$
has equally sized fibers. Hence $\gv_{M',v}=\# X'=\#X \cdot \#(\textrm{fiber above } x) =\gv_{M,v} \cdot \gv_{M',x}$ as required. Similarly, if one
replaces $K$ by $L$, one obtains $\gv_{M',w}=\gv_{M,w} \cdot \gv_{M',x}$.
Hence the required ratio does not depend on the choice of $M$.

From Lemma \ref{1c90} and
Theorem \ref{1c333}ii one obtains $[K_{\hv,x}:K]=\gv_{M,v}$ and $[LK_{\hv,x}:L]=[L_{\hv,x}:L]=\gv_{M,w}$. Hence we have
\begin{align*}
[L_{\hv,x}:K_{\hv,x}]=[LK_{\hv,x}:K_{\hv,x}]= \frac{[LK_{\hv,x}:L]}{[K_{\hv,x}:K]}\cdot [L:K]=\frac{\gv_{M,w}}{\gv_{M,v}}\cdot[L:K]=\nv(w/v). 
\end{align*}

We will now prove the last statement. It is obvious that $\ev$, $\ev_{\tv}$ $\ev_{\wv}$, $\fv$, $\fv_{\sv}$, $\fv_{\iv}$ are multiplicative. If
we show that $\nv$ is multiplicative, it directly follows that $\dv$ and $\dv_{\wv}$ are multiplicative. Hence it is enough to show that $\nv$
is multiplicative. Let $(L',w')$ be a finite extension of $(L,w)$. Let $M$ be a finite normal extension of $K$ containing $L'$. Then one has
\begin{align*}
 \nv(w'/w)\nv(w/v)&=\frac{\gv_{M,w'}}{\gv_{M,w}} \cdot [L':L] \cdot \frac{\gv_{M,w}}{\gv_{M,v}} \cdot [L:K] \\
&= \frac{\gv_{M,w'}}{\gv_{M,v}} [L':K] = \nv(w'/v).
\end{align*}

ii. It is now obvious that $\dv(w/v)$ is well-defined. One has $\dv(w/v)=\frac{[L_{\hv,x}:K_{\hv,x}]}{\ev(w/v)  \fv(w/v)} \in \Z_{\geq 1}$ by Proposition \ref{1c55}, Proposition \ref{1c56} and the multiplicativity of $\ev$ and $\fv$. 

If $L/K$ is normal, one has 
\begin{align*}
\dv(w/v) =\frac{[L:K_{\hv,w}]}{\ev(w/v)\fv(w/v)}=\frac{[L:K_{\vv,w}]}{\ev_{\wv}(w/v)\fv_{\sv}(w/v)} \in \pv_v^{\Z}
\end{align*}
(Lemma \ref{1c90}). Together
with the multiplicativity of $\dv$, this shows $\dv(w/v) \in \pv_v^{\Z_{\geq 0}}$ in general. 
\end{proof}

\begin{proof}[Proof of Proposition \ref{1c854}] \label{1cp3}
 This follows directly from Lemma \ref{1c90} and the definitions.
\end{proof}

\begin{proof}[Proof of Theorem \ref{1c933}] \label{1cp4}
The first equality is easily from the definition, the second follows by definition and the third follows from Proposition \ref{1c4444}ii. 
\end{proof}

\begin{remark} \label{1c754}
Let $(L,w) \supseteq (K,v)$ be an algebraic extension of valued fields. 
Note that immediate implies unramified, unramified implies tame, totally wild implies totally ramified and totally ramified implies local.

Suppose that one of the following hold:
\begin{enumerate}
 \item $w/v$ is immediate and local; 
 \item $w/v$ is unramified and totally ramified;
 \item $w/v$ is tame and totally wild.
\end{enumerate}
Then from Theorem \ref{1c933} it follows that $L=K$.
\end{remark}

\begin{remark}
 Let $(L,w) \supseteq (K,v)$ be an algebraic extension of valued fields which is purely inseparable. Then one easily sees that it is totally wild
(Lemma \ref{1c72}). 
\end{remark}

\begin{proof}[Proof of Theorem \ref{1c802}] \label{1cp5}
The surjectivity of $\pi$ follows directly from the transitivity as in Proposition \ref{1c55} and the extension property as in Proposition \ref{1c143}.

Let $\sigma \in X$. This gives us embeddings $(K,v) \subseteq (L,w) \subseteq (M,x)$. Write $H=\Gal(M/L)$. Then $X$ corresponds to $G/H$.
Let $\mathfrak{T}$ be the set of finite normal subextensions of $M/K$. We first consider $w/v$:

i, ii, iii: $w/v$ is immediate $\iff$ for all $M' \in \mathfrak{T}$ the extension $w|_{M' \cap L}/v$ is immediate $\iff$ for all $M' \in
\mathfrak{T}$ we have $\# \Dv_{x|_{M'},K}=\# \Dv_{x|_{M'},L \cap M'}$ (Proposition \ref{1c854}, look at degrees such as $[K:K_{\hv,x|_{M'}}]=\# \Dv_{x|_{M'},K}$) $\iff$ for all $M' \in \mathfrak{T}$ we have
$\Dv_{x|_{M'},K}=\Dv_{x|_{M'},L \cap M'}=\Dv_{x|_{M'},K} \cap \Aut_{L \cap M'}(M')$ (Theorem \ref{1c333}) $\iff$ for all $M' \in \mathfrak{T}$ we have
$\Dv_{x|_{M'},K}
\subseteq \Aut_{L \cap M'}(M')$  $\iff$ $\Dv_{x,K} \subseteq H$ (Theorem \ref{1c333}) $\iff$ $L \subseteq K_{\hv,x}$. The other
cases are similar.

iv, v, vi: $w/v$ is local $\iff$ for all $M' \in \mathfrak{T}$ the extension $w|_{M' \cap L}/v$ is local $\iff$ for all $M' \in \mathfrak{T}$ we have
$[L \cap M':K][K_{\hv,x|_{M'}}:K]=[(L \cap M')_{\hv,x|_{M'}}:K_{\hv,x|_{M'}}][K_{\hv,x|_{M'}}:K]=[(L \cap M')K_{\hv,x|_{M'}}:K]$ (Proposition
\ref{1c854} and Theorem \ref{1c333}) $\iff$ for all $M' \in \mathfrak{T}$ we have $\Aut_{K}(M')=\Dv_{x|_{M'},K} \Aut_{L \cap M'}(M')$ (as $K_{\hv,x}$
is separable, one can apply Proposition \ref{1c931}) $\iff$ for all $M' \in \mathfrak{T}$ the group $\Dv_{x|_{M'},K}$ acts transitively on
$\Aut_K(M')/\Aut_{L \cap M'}(M')$.

We prove that the last statement is equivalent with $\Dv_{x,K}$ acting transitively on $G/H$. If $\Dv_{x,K}$ acts transitively on $G/H$, then one
easily sees from the surjectivity $G/H \to \Aut_K(M')/\Aut_{L \cap M'}(M')$ for $M' \in \mathfrak{T}$ and Theorem \ref{1c333}, that $\Dv_{x|_{M'},K}$
acts transitively on $\Aut_K(M')/\Aut_{L \cap M'}(M')$. Conversely, given $u,v \in G/H$, consider the projective system of non-empty finite sets
which at level $M'$ consists of those elements mapping $u|_{M'}$ to $v|_{M'}$. One can easily show that this set is non-empty and deduce the result.
The other cases are similar.

Now consider $x/w$:

vii, viii, ix: $x/w$ is immediate $\iff$ $M \subseteq L_{\hv,x}=LK_{\hv,x}$ (using i and Theorem \ref{1c333}) $\iff$ $M=LK_{\hv,x}$. The latter is equivalent to $[\sigma(L):K]_{\iv}=[M:K]_{\iv}$ and $G_{\sigma} \cap \Dv_{x,K}=0$ by Galois theory. 
The proofs of viii and ix are similar.

x, xi, xii: $x/w$ is local $\iff$ $M$ and $L_{\hv,x}=LK_{\hv,x}$ are linearly disjoint over $L$ (using ii and Theorem \ref{1c333}) $\iff$ $K_{\hv,x}
\subseteq L$ iff $G_{\sigma} \subseteq \Dv_{x,K}$. The proofs of xi and xii are similar.

Consider the last statements.

xiii: This follows from i and the surjectivity of $\pi$.

xiv: $w$ is totally split in $M$ $\iff$ $M/\sigma(L)$ is separable and for all $g \in G$ we have $G_{\sigma} \cap \Dv_{g(x),K}=0$ (vii and Theorem \ref{1c333})  iff $M/L$ is separable and only the trivial element of $G_{\sigma}$ is conjugate to $\Dv_{x,K}$ (Theorem \ref{1c333}).
\end{proof}

\begin{proof}[Proof of Corollary \ref{1c1432}] \label{1cp6}
Let $(E,x')$ be a normal extension of $(K,v)$ extending the valued field $(LL',x)$. 

i, ii, iii, iv: Assume that $w/v$ is immediate. Theorem \ref{1c802} gives us that $L \subseteq K_{h,x'}$. But then we have $LL' \subseteq
L'K_{h,x'}=L'_{h,x'}$ (Theorem \ref{1c333}). Hence Theorem \ref{1c802} shows that $x/w'$ is immediate. The proofs for the other cases are similar.
\end{proof}

\begin{remark}
 Statements as in Corollary \ref{1c1432} are false for local, totally ramified or totally wild extensions. Here is an example from algebraic number
theory. Let $K=\Q$, $L=\Q(\sqrt{7})$ and $L'=\Q(\sqrt{-1})$ and look at the primes above $2$. In this case $L/K$ and $L'/K$ is totally wild (and hence local and totally ramified). In the extension $L''=\Q(\sqrt{-7})$ of $\Q$ the prime $2$ splits. Hence in the extension $LL'/K=\Q(\sqrt{7},\sqrt{-1})/\Q(\sqrt{-1})$ the prime above $2$ splits. The extension $LL'/K$ is not local.
\end{remark}

\begin{proof}[Proof of Corollary \ref{1c842}] \label{1cp7}
Let $(M,x)$ be a normal extension of $(K,v)$ containing $(L,w)$.

For the immediate case, we have the following: $L \subseteq K_{\hv,x}$ $\iff$ $K' \subseteq K_{\hv,x}$ and $L \subseteq K'_{\hv,x}=K'
K_{\hv,x}$ (Theorem \ref{1c333}). The result follows from Theorem \ref{1c802}. The unramified and tame cases are similar.

Now consider the local case. One has: $L \otimes_K K_{\hv,x}$ is a domain $\iff$ $L \otimes_{K'} K'_{\hv,x}$ and $K' \otimes_K K_{\hv,x}$ are
domains. Assume first that $K' \otimes_K K_{\hv,x}$ is a domain. Observe that
\begin{align*}
L \otimes_{K'} K'_{\hv,x} &= L \otimes_{K'} \left( K' K_{\hv,x} \right) \\
&= L \otimes_{K'} K' \otimes_K K_{\hv,x} \\
&= L \otimes_K K_{\hv,x}
\end{align*}
(Theorem \ref{1c333}).
This directly proves $\limplies$. 
The implication $\implies$ follows from $K' \otimes_K K_{\hv,x} \subseteq L \otimes_K K_{\hv,x}$ and the above observation observation.

The result follows from Theorem \ref{1c802}. The totally ramified and totally wild cases are similar.

\end{proof}

\begin{proof}[Proof of Corollary \ref{1c803}] \label{1cp8}
Pick an extension $(M,x) \supseteq (L,w) \supseteq (K,v)$ such that $M/K$ is normal. Using Theorem \ref{1c802} we see $L_1=K_{\hv,x} \cap L$,
$L_2=K_{\iv,x} \cap L$ and $L_3=K_{\vv,x} \cap L$.

We will now construct a minimal local subextension. Assume that $L'$ is a field such that $w/w|_{L'}$ is local. Then $w/w|_{L'_{K,\sep}}$ is also local (Lemma \ref{1c72}). Hence we can replace $L$ by $L_{K,\sep}$. Using Theorem \ref{1c802} and Theorem \ref{1c333}, we see
that
we need to find the smallest intermediate field $L'$ of $L_{K,\sep}/K$ such that $L_{K,\sep}$ and $L'_{\hv,x}=L'K_{\hv,x}$ are
linearly disjoint over $L'$. Such a field exists by Theorem \ref{1c432} and it is denoted by $L_{K,\sep} \fod
K_{\hv,x}$. The proofs of the other cases are similar when $\Dv$ is replaced by $\Iv$ respectively $\Vv$.

We will now prove that $L_1 \subseteq L_4$. As $w|_{L_1}/v$ is immediate, it follows that the extension $w|_{L_1L_4}/w|_{L_4}$ is immediate (Corollary \ref{1c1432}). Hence $L_1L_4/L_4$ is immediate and local.
From Remark \ref{1c754} it follows that $L_1L_4=L_4$, that is, $L_1 \subseteq L_4$. The inclusions $L_2 \subseteq L_5$ and $L_3 \subseteq L_6$ follow in a similar manner.

Assume that $\gv_{M,w}=1$ for some normal extension $M/K$ containing $L$. Note that $\gv_{M,w}=1$ is equivalent to $H=\Aut_L(M) \subseteq \Dv_{x,K}$ (Theorem \ref{1c802}x). From Theorem \ref{1c802} it
follows that we need
to
show that
$\Dv_{x,K} \to
\langle
H,\Dv_{x,K} \rangle/H$, $\Iv_{x,K} \to \langle H,\Iv_{x,K} \rangle /H$ and $\Vv_{x,K} \to \langle H,\Vv_{x,K} \rangle /H$ are surjective. The
surjectivity of
the first map is obvious, and the surjectivity of the second and third map is implied by the normality of $\Iv_{x,K}$ respectively $\Vv_{x,K}$ in
$\Dv_{x,K}$ (Theorem \ref{1c333}).
\end{proof}

Actually, one can make the diagram a bit bigger. For $i=1,2,3$ we define $L_i'$ to be the intersection of the $L_i$ while varying over the extensions
of $v$ to $L$. Similarly, for $i=4,5,6$ we define $L_i'$ to be the compositum of the $L_i$ while varying over the extensions of $v$ to $L$. For
example, $L_1'$ is the maximal extension such that $v$ is totally split. If we put $L_7'=L_{K,\sep}$, we get the following diagram,
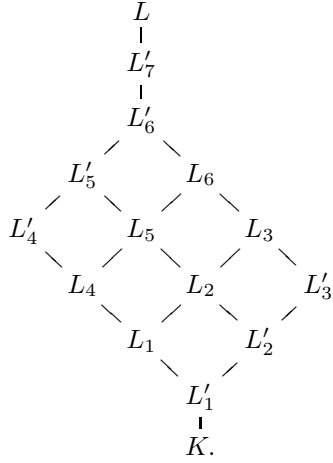
\[
\xymatrix@=0.5em{
 & & L & & &\\
 & & L_7' \ar@{-}[u] & & & \\
 & & L_6' \ar@{-}[u] & & &\\
 & L_5' \ar@{-}[ru] & & L_6 \ar@{-}[lu] & &\\
L_4' \ar@{-}[ru]& & L_5 \ar@{-}[ru] \ar@{-}[lu] & & L_3 \ar@{-}[lu] &\\
 & L_4 \ar@{-}[ru]\ar@{-}[lu] & & L_2 \ar@{-}[ru] \ar@{-}[lu] & &  L_3' \ar@{-}[lu] \\
 & & L_1 \ar@{-}[lu] \ar@{-}[ru]&  & L_2' \ar@{-}[lu] \ar@{-}[ru]&\\
 & & & L_1' \ar@{-}[lu]\ar@{-}[ru]& &\\
 & & & K. \ar@{-}[u] & &\\
}  
\]

\begin{proof}[Proof of Corollary \ref{1c14324}] \label{1cp9}
We will prove ii. The other proofs are similar. 

Assume that $\gv_{M,w}=1$. Assume that $w/v$ is not totally ramified. Then one has $L_2=L_5$ in
Corollary \ref{1c803}. And hence there is a non-trivial unramified extension in $L/K$. But then $L_5 L'/L'$ is unramified and non-trivial (Corollary
\ref{1c1432} respectively $L \cap L'=K$). Contradiction. 
\end{proof}

\begin{example}
In Corollary \ref{1c14324} it is not enough to require that $L \cap L'=K$. Here is an example where all three statements are false.
Consider the extension $L=\Q(\alpha)$ of $\Q$ where $\alpha$ is a root of $x(x-1)^2+2$. Well-known techniques show that there are two primes above
$(2)$, namely
$\pa=(2,\alpha)$ and $\qa=(2,\alpha-1)$. One has $(2)=\pa \qa^2$. It follows that $\Q(\alpha)/\Q$ is not Galois. Hence the Galois closure $M$ of this
extension has group $S_3$. Let $\overline{\alpha}$ be another root in this Galois closure and let $L'=\Q(\overline{\alpha})$. Then the prime $(2)$
has the same splitting behavior in $L'/\Q$ as in $L/\Q$, say $(2)=\pa' \qa'^2$. Note that $LL'=M$, $L \cap L'=\Q$. Furthermore, we know the splitting
behavior of $(2)$ in $M$: $(2)=\pa_1^2 \pa_2^2 \pa_3^2$ where there is just one prime above $\pa$ respectively $\pa'$, which is totally wild, and
there are two primes above $\qa$ respectively $\qa'$. Say that $\pa_1$ lies above $\pa'$. Then $\pa_1/\pa'$ is totally wild, but $\pa_1|_L/2\Z$ is
not even local. Hence statements i, ii and iii of Corollary \ref{1c14324} do not hold in this case.

The same example shows that in Corollary \ref{1c803} it is not necessarily true that $L_1=L_4$, $L_2=L_5$ or $L_3=L_6$. Indeed, for
the extension $(\Q(\alpha),\qa)/(\Q,2\Z)$ we have $L_1=L_2=L_3=\Q$ and $L_4=L_5=L_6=\Q(\alpha)$.
\end{example}

\begin{example}
 Let $K$ be a field with field extensions $L,L'$ inside a field $M$. Assume that $L \cap L'=K$ and that $M/L'$ is purely inseparable. Then $L/K$ is
purely inseparable. Indeed, if $\chart(k)=p>0$, then for $x \in L$ there is $n \in \Z_{\geq 1}$ with $x^{p^n} \in L \cap L'=K$.

This statement also follows from our general theory. Consider the trivial valuation on $K$, that is, $K$ is the valuation ring. This valuation has a
unique valuation to any algebraic field extension of $K$. Furthermore, $M/L'$ is totally wild. Hence from Corollary \ref{1c14324} it follows that
$L/K$ is totally wild and the result follows.
\end{example}

\begin{proposition} \label{1c900}
 Let $(K,v)$ be a valued field and let $L$ be an algebraic extension of $K$. Let $(M,x) \supseteq (K,v)$ be a normal extension of valued fields with
group $G=\Aut_K(L)$ such that the $G$-set $X_L=\Hom_K(L,M)$ is not empty. Then for any intermediate extension $L'$ of $L/K$ we have the
following commutative diagram, where the maps are the natural maps:
\[
\xymatrix{
G \times X_L \ar[r] \ar@{->>}[d] & X_L \ar@{->>}[d] \\
G \times X_{L'} \ar[r] & X_{L'}. 
}  
\]
The map 
\begin{align*}
 \varphi \colon \Dv_{x,K} \backslash X_L &\to \{w \textrm{ of }L\textrm{ extending }v\} \\
\Dv_{x,K} \sigma &\mapsto w \textrm{ s.t. } \mathcal{O}_w=\sigma^{-1} ( \mathcal{O}_x \cap \sigma(L))
\end{align*}
is a bijection of sets.
Furthermore, for $\sigma \in X_L$ we have the following bijections:
\begin{align*}
  \Dv_{x,K}\sigma &\to \Hom_{K_{\hv,x}}(\sigma(L)_{\hv,x},M) \\
           t\sigma &\mapsto t|_{\sigma(L)_{\hv,x}}, 
\end{align*}
\begin{align*}
  \Iv_{x,K}\sigma &\to \Hom_{K_{\iv,x}}(\sigma(L)_{\iv,x},M) \\
           t\sigma &\mapsto t|_{\sigma(L)_{\iv,x}}, 
\end{align*}
and 
\begin{align*}
  \Vv_{x,K}\sigma &\to \Hom_{K_{\vv,x}}(\sigma(L)_{\vv,x},M) \\
           t\sigma &\mapsto t|_{\sigma(L)_{\vv,x}}.
\end{align*}
\end{proposition}
\begin{proof}
 The commutativity of the diagram is obvious. 

Define $\varphi' \colon X_L \to \{w \textrm{ of }L\textrm{ extending }v\}$ by putting $\sigma \mapsto w \textrm{ s.t. } \mathcal{O}_w=\sigma^{-1} ( \mathcal{O}_x \cap \sigma(L))$. 
One should think of $\varphi'$ as mapping an embedding $L \subseteq M$ to the restriction of $x$ to
$L$. The surjectivity is part of Theorem \ref{1c802}. Suppose
$\varphi'(s)=\varphi'(t)$. There exists $h \in G$ such that $h t=s$. But then by Proposition \ref{1c55} there exists $g \in \Aut_{s(L)}(M)$
with $gh(x)=x$, that is, $gh \in \Dv_{x,s(L)} \subseteq \Dv_{x,K}$. We have $ght=ht=s$. It is obvious that $\varphi'(\Dv_{x,K}s)=\varphi'(s)$. This
shows that the map is a bijection. 

We will show that the map $\Dv_{x,K}s \to \Hom_{K_{\hv,x}}(s(L)_{\hv,x},M)$ is a bijection. The other cases are similar. Suppose we have $\tau \in
\Hom_{K_{\hv,x}}(s(L)_{\hv,x},M)$. Then we can extend it to a morphism $\tau' \in \Aut_{K_{\hv,x}}(M)=\Dv_{x,K}$ and $\tau' \mapsto \tau$.
\end{proof}

\begin{proof}[Proof of Proposition \ref{1c4343}] \label{1cp10}
 The first statement directly follows from Proposition \ref{1c900}. The last statements follows from Proposition \ref{1c900} and Proposition
\ref{1c854} and the separability of $L/K$. 
\end{proof}

\begin{proof}[Proof of Corollary \ref{1c388}] \label{1cp11}
From Proposition \ref{1c4343} one sees that the set of valuations with the given properties is in bijection with the set of orbits of $X$ under
$\Dv_{x,K}$ such that the length of such an orbit is equal to the length of the orbit under $\Iv_{x,K}$. And this easily translates to the required
statement.
\end{proof}

\subsection{Finding extensions explicitly}

\begin{proposition} \label{1c338}
 Let $(K,v)$ be a valued field and let $L/K$ be a finite extension. Pick $a \in L$ which is integral over $\mathcal{O}_v$ with minimal polynomial $f \in \mathcal{O}_v[x]$. Suppose that $\overline{f}=\prod_{i=1}^m
f_i^{n_i} \in \kv_v[x]$ where the $f_i$ are monic irreducible and pairwise distinct. Then the following hold:
\begin{enumerate}
 \item for $i=1,\ldots,m$ there are pairwise distinct valuations $w_i$ on $L$ with $\fv(w_i/v) \geq \deg(f_i)$;
 \item if $\overline{f}$ is separable, then the $w_i$ are all valuations extending $v$ to $L$ and one has $\fv(w_i/v)=\deg(f_i)$,
$\ev(w_i/v)=\dv(w_i/v)=1$.
\end{enumerate}
\end{proposition}
\begin{proof}
Notice that $f K[x] \cap R[x]=f R[x]$ as $f$ is monic. Statement i follows directly from Proposition \ref{1c143} and Proposition \ref{1c434}. Statement ii
follows from i and Theorem \ref{1c933}.
\end{proof}

If the valuation in the above statement is discrete, one can say a bit more. See for example \cite[Theorem 3.3.7]{ST}. 

\section{Defects in the discrete case} 

In this section we will give examples of defects and show that under certain circumstances, defects do not occur. This section is quite different
from the other sections in this article, but we felt it was needed to show the reader that defects are not necessarily a defect of our theory.

We start with an example where there is a defect.

\begin{example}
Let $(L,w) \supseteq (K,v)$ be a finite purely inseparable extension of valued fields where $v$ is discrete, that is, $\Delta_v \cong \Z$. Then one
can have $\dv(w/v)>1$. Let $p$ be a prime number. Consider $\F_p(t) \subseteq \F_p((t))$ with the valuation $w_0$ on $\F_p((t))$ with $w_0(t)=1$.
Let $v_0$ be its restriction to $\F_p(t)$. Then we have $\Delta_{v_0}=\Delta_{w_0}$ and $\kv_{v_0}=\kv_{w_0}$. Let $s \in
\F_p((t))$ be transcendental
over $\F_p(t)$ (such $s$ exist, because $\overline{\F_p(t)}$ is countable, and $\F_p((t))$ is uncountable) and consider $K=\F_p(t,s^p) \subseteq
\F_p(t,s)=L$,
with restricted valuations $v$ respectively $w$. This is a purely inseparable extension of degree $p$ with the property that $\gv_{L,v}=\ev(w/v)=\fv(w/v)=1$.
From Proposition \ref{1c55} and Theorem \ref{1c933} we conclude $\dv(w/v)=p$.
\end{example}

We will show that in certain cases, there is no defect. We use the following lemma.

\begin{lemma} \label{1c323}
Let $k$ be a field and let $A$ be a localization at a multiplicative set of a finitely generated $k$-algebra which is a domain. Put $K=Q(A)$ and let
$L/K$ be a finite
extension of fields. Then
the integral closure $\overline{A}$ of $A$ in $L$ is finite as $A$-module.
\end{lemma}
\begin{proof}
Assume first that $A$ is finitely generated as $k$-algebra.

Notice that it is enough to prove the statement for a finite extension of $L$.  Indeed, a
finitely generated module over a noetherian ring is a noetherian module (\cite[Proposition 6.5]{AT}), and hence all submodules are finitely
generated. 

Noether normalization, \cite[Proposition 2.1.9]{LIU}, tells us that $A$ is finite over a polynomial ring $A'=k[x_1,\ldots,x_n]$ with quotient field
$K'$. We show that the integral closure of $A'$ in $L$, which is $\overline{A}$, is a finite $A'$-module and hence a finite $A$-module.
This reduces to the case where $A=k[x_1,\ldots,x_n]$. 

We will start enlarging $L$. First enlarge it such that $L/K$ is normal. We can split $L/K$ into a tower $L \supseteq L' \supseteq K$ where $L'/K$ is
purely inseparable and $L/L'$ is separable. Hence we are reduced to proving the following two cases:
\begin{enumerate}
\item $L/K$ separable;
\item $L/K$ purely inseparable and $A=k[x_1,\ldots,x_n]$.
\end{enumerate}

Assume that $L/K$ is separable. Let $y_1,\ldots,y_m$ be a basis of $L/K$ with $y_i \in \overline{A}$. Let $y_1',\ldots,y_m'$ be a dual basis of $L/K$
with respect to the trace. Then it follows that $Ay_1 \oplus \ldots \oplus Ay_n \subseteq \overline{A} \subseteq Ay_1' \oplus \ldots \oplus Ay_n'$.
Note that $Ay_1' \oplus \ldots \oplus Ay_n'$ is a finitely generated module over a noetherian ring, and hence a finite $A$-module. It follows that
$\overline{A}$ is a finite $A$-module.

Assume that $L/K$ is purely inseparable and $A=k[x_1,\ldots,x_n]$. Since $L/K$ is finite, we see that $L$ is contained in $L'=l(x_1^{-p^d},\ldots,x_n^{-p^d})$ for some $d \in
\Z_{\geq 0}$ large enough and $l$ a finite (purely inseparable) extension of $k$. Replace $L$ by $L'$. Notice that
$A'=l[x_1^{-p^d},\ldots,x_n^{-p^d}]$ is integral over $A$ and it is integrally closed. Hence the integral closure of $A$ in $L$ is $A'$, and it is
finite over $A$. 

We will now treat the general case. Write $A=S^{-1}B$ where $B$ is a finitely generated $k$-algebra and $S$ a multiplicative set. From \cite[Proposition 5.12]{AT} we obtain $\overline{A}=S^{-1} \overline{B}$. We have shown that $B$ is a finite $A$-module and hence $\overline{A}$ is a
finite $A$-module.
\end{proof}

\begin{proposition} \label{1c888}
Let $(L,w) \supseteq (K,v)$ be a finite extension of valued fields. Suppose one of the following hold:
\begin{enumerate}
\item $L/K$ is separable and $\Delta_v \cong \Z$;
\item $\mathcal{O}_v$ contains a field $k$, $\mathcal{O}_v \neq K$, $K$ finitely generated over $k$, and $\trdeg_k(K)=1$.
\end{enumerate}
Then we have $\Delta_v \cong \Z$ and $\dv(w/v)=1$.
\end{proposition}
\begin{proof}

First we prove that in the second case we also have $\Delta_v \cong \Z$. Let $x \in \mathcal{O}_v$ transcendental over $k$.  Then $\pa=k[x] \cap \m_v$
is a prime ideal. If it is zero, then $\mathcal{O}_v \supseteq k(x)$ and since $K$ is finite over $k(x)$, it follows that $\mathcal{O}_v=K$,
contradiction. Hence $\mathcal{O}_v \cap k(x)=k[x]_{\pa}$ (this follows since we know all valuations on $k(x)$ which are trivial on $k$). Notice that $k[x]_{\pa}$ is a discrete valuation ring, and hence the
same follows for $\mathcal{O}_v$ (as $\ev$ is finite). Replace $\mathcal{O}_v$ by $k[x]_{\pa}$ and $K$ by $k(x)$ in this case. We will show the
statement about $\dv(w/v)$ for the bigger extension, and the result about $\dv(w/v)$ follows from multiplicativity. 

Now we will consider both cases at once. Let $\mathcal{O}$ be the integral closure of $\mathcal{O}_v$ in $L$. Assume first that $\mathcal{O}$ is a
finitely generated $\mathcal{O}_v$-module. 
Then one easily sees that $\mathcal{O}$ is a free $\mathcal{O}_v$-module of rank $[L:K]$, since $\mathcal{O}_v$ is a discrete valuation ring. Consider
$\mathcal{O}/\m_v \mathcal{O}$, which is isomorphic to $\prod_{w|v} \mathcal{O}_w/\m_v \mathcal{O}_w$ (Proposition \ref{1c434}, in combination with
theorems on artinian rings from \cite{AT}). Notice that $\mathcal{O}_w/\m_v \mathcal{O}_w$ is a vector space over $\kv_v$ of dimension
$\ev(w/v)\fv(w/v)$ and
the result follows.

Hence we are finished if we can show that $\mathcal{O}$ is finite over $\mathcal{O}_v$. In the first case, this follows directly from the trace
pairing. In the second case, use Lemma \ref{1c323}. 
\end{proof}

We will finish this section by giving an example of a separable extension of valued fields which has a defect. We start with the following lemma,
which goes back to \cite{SCH}. We follow a proof from \cite{STE2}.

\begin{lemma} \label{1c755}
Let $K$ be a field and $n \in \Z_{\geq 1}$ be an integer coprime with $\chart{K}$. Let $w$ be the number of $n$-th roots of unity in $K$. Let $L$ be
the splitting field of $x^n-a \in K[x]$. Then one has: $L/K$ is abelian iff $a^w \in K^n$. 
\end{lemma}
\begin{proof}
We may assume $a \neq 0$.

$\implies$: Fix $\alpha \in L$ with $\alpha^n=a$ and let $\zeta_n$ be an $n$-th root of unity. Let $\sigma \in G=\Gal(L/K)$. Write
$\sigma(\zeta_n)=\zeta_n^{k(\sigma)}$. For $\tau \in G$ one has
\begin{align*}
 \frac{\tau \sigma (\alpha)}{\sigma(\alpha)} =  \frac{\sigma \tau (\alpha)}{\sigma(\alpha)} = \sigma \left( \frac{\tau(\alpha)}{\alpha} \right)=
\left( \frac{\tau(\alpha)}{\alpha} \right)^{k(\sigma)} = \frac{\tau(\alpha^{k(\sigma)})}{\alpha^{k(\sigma)}}. 
\end{align*}
Hence $\alpha^{k(\sigma)}/\sigma(\alpha)$ is fixed by $\tau$ and hence lies in $K$. Its $n$-th power is $a^{k(\sigma)-1} \in K^n$. Let
$r$ be the greatest common divisor of $n$ and $k(\sigma)-1$ for $\sigma \in G$. Then we have $a^r \in K^n$. As $\langle \zeta_n^{n/r} \rangle$ is the set of
$G$-invariant $n$-th
roots of unity, one has $r=w$.

$\limplies$: Suppose $a^w=b^n$ for some $b \in K$. One has $K \subseteq L \subseteq L'=K(b^{1/w},\zeta_{nw})$. Notice that $L'/K$ is abelian and hence
$L/K$ is abelian.
\end{proof}

\begin{remark}
Next we will give an example of a separable extension which has a defect. Let $p$ be a prime and consider the field $\Q_p$ with its standard p-adic valuation. It is well-known that this valuation has a unique extension to each algebraic extension ($\Q_p$ is \emph{henselian}). Let $L$ be
the maximal tamely ramified extension of $\Q_p$. Put $L'=L(\zeta_{p^i}: i \in \Z_{\geq 1})$. We claim that for any finite extension $L''/L'$ we have
$\dv(L''/L')=[L'':L']$ (we do not specify the valuations, since they are unique). Indeed, from the construction one easily sees that $\ev=\fv=1$ (the
residue field of $L$ is already algebraically closed, and the value group of $L'$ is $\Q$) and as the extension is unique, the degree is equal to the
defect. We will now find a non-trivial extension $L''/L'$. We claim that $\sqrt[p^2]{p} \not
\in L'$. Suppose $\sqrt[p^2]{p} \in L'$, then
$L(\sqrt[p^2]{p})/L$ is an abelian Galois extension. Note that $\#\{x \in L: x^{p^2}=1\}=p$, as $\Q_p(\zeta_{p^2})/\Q_p$ is wild. 
Lemma \ref{1c755} gives us $p^p \in L^{p^2}$. But this means that $L/\Q_p$ is not tame, contradiction. Hence we can take $L''=L'(\sqrt[p^2]{p})$.
\end{remark}

\end{document}